\documentclass{scrartcl}

\usepackage[utf8]{inputenc}
\usepackage[T1]{fontenc}

\usepackage{amsmath,amsfonts,amssymb,amsthm,mathrsfs,amsbsy}
\usepackage{wasysym}
\usepackage{minitoc}
\usepackage{endnotes}
\usepackage[dvipsnames]{xcolor}
\usepackage[a4paper,vmargin={3.5cm,3.5cm},hmargin={2cm,2cm}]{geometry}
\usepackage[font=sf, labelfont={sf,bf}, margin=1cm]{caption}
\usepackage{graphicx,graphics}
\usepackage{epsfig}%pour les eps
\usepackage{latexsym}%encore des symboles
\usepackage{ae,aecompl}

\usepackage[english]{babel}
 \usepackage[colorlinks=true]{hyperref}
\usepackage{pstricks}
\usepackage{enumerate}
\usepackage[shortlabels]{enumitem}
\renewcommand{\leq}{\leqslant}
\renewcommand{\geq}{\geqslant}
\usepackage{mathpazo}
\usepackage[font=sf, labelfont={sf,bf}, margin=1cm]{caption}

\usepackage{amsmath,amssymb,amsthm,mathrsfs}
\usepackage{esint} % double contour integral symbols and all that
\usepackage{upgreek}
\usepackage{dsfont}
\usepackage{footnote}
\usepackage{comment}
\usepackage{enumitem}

\usepackage[backend=bibtex,style=alphabetic,giveninits=true,maxnames=99,maxalphanames=99,autopunct=false]{biblatex}
\usepackage{hyperref}
\hypersetup{colorlinks=true, citecolor=ForestGreen, breaklinks, urlcolor=black, linkcolor=Black}
\usepackage{cleveref}
\AtEveryCite{\color{ForestGreen}}

\usepackage{tikz}
\usetikzlibrary{decorations.pathmorphing}

\newtheorem{thm}{Theorem}
\newtheorem{prop}{Proposition}
\newtheorem{lem}[prop]{Lemma}
\newtheorem{cor}[prop]{Corollary}
\theoremstyle{definition}
\newtheorem{Def}{Definition}
\theoremstyle{remark}
\newtheorem{rem}{Remark}
\usepackage[framemethod=TikZ]{mdframed}

\global\mdfdefinestyle{exampledefault}{%
linecolor=lightgray,linewidth=1pt,%
leftmargin=0.1cm,rightmargin=0.1cm,
}

\newenvironment{framednote}[1]{%
\mdfsetup{%
frametitle={\colorbox{white}{\,#1\,}},
frametitleaboveskip=-\ht\strutbox,
frametitlealignment=\raggedright
}%
\begin{mdframed}[style=exampledefault]
}{\end{mdframed}}

\makeatletter
\newif\iflibus@sansmath
\makeatother
\DeclareFontEncoding{LS1}{}{}
\DeclareFontSubstitution{LS1}{libertinust1math}{m}{n}
\DeclareSymbolFont{LettersLibertinus}{LS1}{libertinust1math}{m}{it}
\DeclareMathSymbol{\tau}{\mathord}{LettersLibertinus}{"1C}

\DeclareMathOperator{\BO}{\textit{O}}
\newcommand\BigO[1]{\BO{\hspace{-0.22em}\left(#1\right)}}
\DeclareMathOperator{\so}{\textit{o}}
\newcommand\smallo[1]{\so{\hspace{-0.15em}\left(#1\right)}}
\DeclareMathOperator{\Esp}{\mathbb{E}}
\newcommand\Expect[1]{\Esp{\hspace{-0.22em}\left[#1\right]}}
\DeclareMathOperator{\Pb}{\mathbb{P}}
\newcommand{\Prob}[1]{\Pb{\hspace{-0.22em}\left(#1\right)}}
\newcommand{\Var}[1]{\mathrm{Var}{\left(#1\right)}}
\newcommand{\indic}{\mathds{1}}
\newcommand\Indic[1]{\indic{_{\{#1\}}}}

\newcommand{\R}{\mathbb{R}}
\newcommand{\Z}{\mathbb{Z}}
\newcommand{\Q}{\mathbb{Q}}
\renewcommand{\P}{\mathbb{P}}
\newcommand{\dinf}{\mathrm{d}}
\renewcommand{\epsilon}{\varepsilon}
\newcommand{\dgh}{ \mathrm{d}_{ \mathrm{GH}}}
\newcommand{\tn}[2]{\tau^{(#1)}_{#2}}
\newcommand{\tc}[1]{\theta_{#1}}
\newcommand{\tcm}[2]{\theta^{(#1)}_{#2}}

\title{Random trees with local catastrophes: the Brownian case}
\author{Ariane Carrance\thanks{Fakultät für Mathematik, Universität Wien \hspace*{\fill} \href{mailto:ariane.carrance@univie.ac.at}{ariane.carrance@univie.ac.at}}, Jérôme Casse\thanks{LMO, Université Paris-Saclay  \hspace*{\fill} \href{mailto:jerome.casse.math@gmail.com}{jerome.casse.math@gmail.com}}, and Nicolas Curien\thanks{LMO, Université Paris-Saclay  \hspace*{\fill} \href{mailto:nicolas.curien@gmail.com}{nicolas.curien@gmail.com}}}
\date{\today}

\begin{document}
\maketitle

\begin{abstract}
We introduce and study a model of plane random trees generalizing the famous Bienaymé--Galton--Watson model by incorporating spatially correlated deaths, which we interpret as local catastrophes. More precisely, given a random variable $(B,H)$ with values in $  \{1,2,3,...\}^2$, given the state of the tree at some generation, the next generation is obtained (informally) by successively deleting $B$ individuals side-by-side and replacing them with $H$ new particles where the samplings are i.i.d. We prove that, in the critical case $ \mathbb{E}[B]= \mathbb{E}[H]$, and under a third moment condition on $B$ and $H$, the random trees coding the genealogy of the population model converge towards the Brownian Continuum Random Forest. Our approach does not use the classical height process or the Łukasiewicz exploration, but instead relies on the stochastic flow point of view introduced by Bertoin \& Le Gall~\cite{bertoinlegall1,bertoinlegall2,bertoinlegall3}.
\end{abstract}

\tableofcontents

\section{Introduction}
The study of discrete and continuous branching processes is a vast domain with many applications to population modeling. We refer for example to \cite{jagers,aldous-crt1,aldous-crt2,aldous-crt3,duquesne-legall-mono,lambert,abraham-delmas-notes,bansaye-simatos} for an overview. The simplest model of a branching process is the famous Bienaymé--Galton--Watson one,  where all particles behave independently of each other and reproduce according to some fixed offspring distribution. In this paper we consider a natural variant of this process where deaths are spatially correlated, which we can interpret as local catastrophes. More precisely, given a law $\rho $ on $\mathbb{Z}_{>0}^2$ with finite first moments called the "\textbf{brick distribution}", we define a discrete population evolution process (which does not satisfy the discrete branching property) as follows. Imagine that at generation $h$, there are $n$ particles alive which are ordered from left to right. We then sample independent identically distributed random variables $(B_i,H_i) : i \geq 1$   with law $\rho$ and form the new generation by starting from the left and iteratively deleting $B_i$ particles at generation $k$ and replacing them by $H_i$ particles at generation $k+1$ until there is no particles left at generation $k$. The successive generations are built similarly. See~\Cref{fig:definition} for an illustration.

\begin{figure}[!h]
\centering
 \includegraphics[width=13cm]{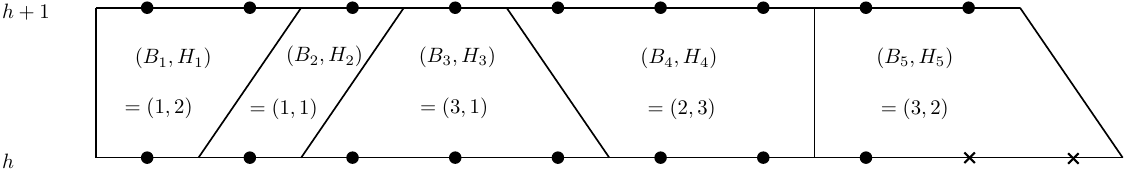}
 \caption{Illustration of the construction of the population evolution process from generation $k$ to $k+1$ given i.i.d.\ bricks sampled according to the law $\rho$. Notice that the last brick overshoot the population at generation $k$: we put crosses for fictitious individuals.}\label{fig:definition}
 \end{figure}
 
The random variables $(B_i,H_i)$ will be interpreted as   \textbf{bricks} with bottom length $B_i$ and top length $H_i$. Those bricks have flexible vertical parts and they are then stacked horizontally and "pushed to the left" on top of generation $k$ to obtain generation $k+1$.   The reader may have noticed that in this construction, the last brick may overshoot the population at generation $k$ and it is actually more convenient to directly define the model in an infinite population setting, where at each generation $k \geq 0$, we have an infinite number of particles indexed by $ \mathbb{Z}$. To preserve an invariance by translation, we now construct generation $k+1$ given generation $k$ using a bi-infinite sequence of ``bricks''
 $$\big((B_i,H_i) : i \leq -1 \big), \quad  ({B_0}^*,H_0), \quad \big((B_i,H_i)  : i \geq 1\big)$$ where $(B_i,H_i) : i \in \mathbb{Z}\backslash \{0\}$ are i.i.d. of law $\rho$ and $({B_0}^*,H_0)$  is the size-biasing of $(B,H)$ by the variable $B$ (we assumed that $B$ has finite expectation). The bricks are then stacked horizontally in the most obvious fashion and pinpointed by choosing a uniform edge in the bottom of $( {B_0}^*, H_0)$ and putting it on the root edge of the generation $k$. The root edge at generation $k+1$ is the left-most of the $H_0$ edges of this brick (see~\Cref{subsec:defs} for a more precise definition). We can then iterate and stack horizontally  i.i.d. rows obtained this way and get a \textbf{brick wall}\footnote{All in all, you're just another brick in the wall}.

It might not be obvious, but this brick wall actually encodes a forest of random trees $(T^\uparrow_i : i \in \mathbb{Z})$ obtained by adding the edges from the middle of the bottom-right edge of each brick to each middle of its top edges, see Figure \ref{fig:bricksall}. These trees, whose vertices are the points of $ (\mathbb{Z} + \frac{1}{2}) \times \mathbb{Z}_{\geq 0}$, are "ascending", and we can define a dual family of "descending" trees using a similar construction, see~\Cref{subsec:defs} for details. Beware, in general our random trees $T^\uparrow_i$ are not i.i.d.\ since they interact through their boundary bricks. 

It should be clear to the reader that the natural criticality assumption is 
\begin{displaymath}
  \mathbb{E}[B] = \mathbb{E}[H], \qquad ( \mathbf{critical}),
\end{displaymath}
which we will assume in these pages. Our main result is that despite their dependencies, the random  trees $(T^\uparrow_i : i \in \mathbb{Z})$ are  in the same universality class as critical Bienaymé--Galton--Watson trees with finite variance. To state the theorem properly let us consider the infinite tree $ \underline{F}^{\uparrow}$ obtained by grafting the forest of trees $(T^\uparrow_i : i \in \mathbb{Z})$ on the bi-infinite line $ \mathbb{Z}$. This random tree will be pointed at the origin $0$ and has almost surely two ends. We denote by $\lambda \cdot  \underline{F}^{\uparrow}$ the random pointed metric space whose point set is the vertex set of $ \underline{F}^{\uparrow}$ endowed with the graph distance $ \mathrm{d_{gr}}$ renormalized by $\lambda >0$. Its scaling limit will be given by the infinite Brownian forest $ \mathcal{F}_{\sigma}$ with parameter $\sigma>0$ obtained by grafting on the line $ \mathbb{R}$ a Poisson point process of random trees with intensity $ \sigma^{-2} \cdot \mathrm{d}t \cdot \mathbf{n}( \mathrm{d}\tau)$ where $ \mathbf{n}$ is Itô's excursion measure on points continuum random trees (CRT), see Section \ref{sec:forest} for details. This infinite continuum random tree is pointed at $0$ and called the Brownian forest with parameter $\sigma$. 

\begin{thm}[Convergence towards the Brownian forest] 
  Suppose that the critical condition holds: $\mathbb{E}[B] = \mathbb{E}[H]$, and that $\mathbb{E}[B^3], \mathbb{E}[H^3] < \infty$ and $\Prob{B=H} < 1$. Let us introduce 
$$ \sigma^2 = \frac{1}{ \mathbb{E}[B]}\Expect{(B-H)^2}.$$
Then we have the following convergence in law for the (local) pointed Gromov--Hausdorff convergence 
$$ \lambda \cdot \underline{F}^{\uparrow} \xrightarrow[\lambda \to0]{(d)}  \mathcal{F}_{\sigma}.$$
\label{thm:main}
\end{thm}

To obtain this, we prove a general result on discrete martingales that applies to our case. This kind of statement (that might be of independent interest) might be ``classical'', but since we have not been able to locate a precise reference for it, we provide a complete proof.

\begin{thm}[Convergence to a Feller diffusion and Kolmogorov's estimate]
\label{thm:general-martingale-diffusion-cv-intro-intro}
For $n \in \Z_{\geq 0}$, let $(A_k^{(n)})_{k \geq 0}$ be a sequence of discrete-time homogeneous, $\Z_{\geq 0}$-valued Markov chains, started from $A_0^{(n)}=n$. Assume that, for each $n \geq 1$, the process  $(A^{(n)}_k)_{k \geq 0}$ is also a martingale, and that there exist constants $\sigma,s>0$ such that the following conditions are satisfied:
\begin{align}
\label{eq:cond-for-feller-cv0}
\Expect{(\Delta A_{1}^{(n)})^2 }&\geq s \ \ \ \forall \, n\geq1\\
\label{eq:cond-for-feller-cv1}
\Expect{(\Delta A_{1}^{(n)})^2 }&\underset{n \to \infty}{\sim} \sigma^2 n\\
\frac{\Delta A_{1}^{(n)}}{\sqrt{n}} &\xrightarrow[n \to \infty]{(d)} \mathcal{N}(0,\sigma^2).
\label{eq:cond-for-feller-cv2}
\end{align}
Let $\tn{n}{0}$ be the extinction time of $(A_k^{(n)})_{k \geq 0}$, i.e. $\tn{n}{0}=\min\{k \geq 1 |A_k^{(n)}=0\}$. Then we have the following joint convergence in distribution:
\begin{equation}
\Bigg(\Bigg(\frac{A_{\lfloor nt \rfloor}^{(n)}}{n} \Bigg)_{t\geq 0},\frac{\tn{n}{0}}{n}\Bigg) \xrightarrow[n \to \infty]{(d)} \Big(\big(X_t\big)_{t\geq 0},\tc{0} \Big),
\label{eq:joint-cv-pop-slice-ext-time}
\end{equation}
where $(X_t)_{t\geq 0}$ is a Feller diffusion, i.e.\ the unique weak solution to 
\begin{equation}
\dinf X_t=\sigma\sqrt{X_t}\dinf W_t
\label{eq:def-Feller}
\end{equation}
with initial condition $X_0=1$, for $(W_t)_{t \geq 0}$ a standard Brownian motion, and where $\tc{0} = \inf \{ t \geq 0 : X_{t} =0\}$ is the extinction time of $X$.

 Furthermore, we have the following ``Kolmogorov estimate''
  \begin{eqnarray} \label{eq:survival-one-tree} \mathbb{P}(\tau^{{(1)}}_{0} \geq n) \sim \frac{2}{\sigma^{2} n}, \quad \mbox{ as } n \to \infty.  \end{eqnarray}
\end{thm}

This result generalizes the case where $(A_{k}^{(n)})_{k \geq 0}$ is a critical branching process with finite variance started from $n$ individuals, for which those results are well-known and in particular \eqref{eq:survival-one-tree} is sometimes called ``Kolmogorov's estimate'', see \cite{Kolmogorov1928}.\par

In our model of locally correlated trees, the discrete martingales $(A_k^{(n)})_{k \geq 0}$ that satisfy the hypotheses of this theorem are the number of individuals descended from the interval $[0,n[$ at time $0$, that are alive at time $k$.

\paragraph{Proof ideas and extensions.} The lack of independence properties in our model compared to the standard Bienaymé--Galton--Watson makes it more cumbersome to use standard tools to prove convergence of trees, such as the contour function or the Łukasiewicz exploration. We cope with this problem by describing our forest of random trees by their discrete  \textbf{stochastic genealogical flows}. In a nutshell, this idea brought to light in the continuous setting by Bertoin \& Le Gall~\cite{bertoinlegall1,bertoinlegall2,bertoinlegall3}  boils down to controlling the evolution of a few subpopulations in our brick wall model (or equivalently by the subforest spanned by a few vertices in $T_i^\downarrow$). Indeed, despite the lack of independence in our model, the tools available in stochastic analysis are robust enough to prove that the offsprings of different subpopulations evolve, in the scaling limit, as \textbf{independent Feller diffusions} (see Proposition \ref{coro:cv-of-tuple-of-pop-slices}). Combined by a tightness estimate (Proposition \ref{eq:cv-of-epsilon-meshing-unif-in-n}) this enables us to control the large scale geometry of our forest of trees using just a few subpopulations' evolution. We hope that the method developed in these pages could lead to other developments, in particular to the case of branching processes in varying environment, see \cite{conchonkerjan2023scaling} for a very recent result in that direction.

\paragraph{Plan of the paper.} In~\Cref{sec: def-prelim-results}, we define in more detail our brick wall model and explain how Bienaymé--Galton--Watson trees are recovered as a special case. We recall the well-known convergence of Bienaymé--Galton--Watson trees with finite variance to the Brownian forest (see~\Cref{thm:aldous}). We then prove that, for critical brick wall models with finite second moments, subpopulations are martingales. In the rest of the paper, we restrict ourselves to critical models with finite third moments. To prove that the trees of these models also converge to a Brownian forest, we start by focusing on the evolution of subpopulations in~\Cref{sec:fidi} where we prove Theorem~\ref{thm:general-martingale-diffusion-cv-intro-intro} before applying it to our model. We first tackle the behavior of large populations after one time step, and we prove in particular variance estimates and a CLT. Then, we establish that the scaling limit of a finite tuple of large subpopulations is a tuple of independent Feller diffusions, and that this convergence holds jointly with that of the corresponding extinction times. This enables us to prove important uniform controls, which are the key to the tightness estimates. \Cref{sec:forest} is devoted to the proof of~\Cref{thm:main}. Finally, in~\Cref{sec:extensions}, we discuss two directions in which to extend our model: one where we consider heavy-tailed brick distributions, and one where the events of birth and death occur at random times given by exponential clocks, rather than at integer times.\\

\noindent \textbf{Acknowledgements.} We are indebted to Camille Coron for many enlightening discussions about this work in its early days. N.C. is supported by SuperGrandMa, the ERC Consolidator Grant 101087572. A.C. is fully supported by the Austrian Science Fund (FWF) 10.55776/F1002.

\section{Definitions and preliminary results}
\label{sec: def-prelim-results}

In this section, we introduce our brick wall model in details, and show how it generalizes Bienaymé--Galton--Watson processes. We then encode it using its primal and dual forests and show that for all critical brick wall models with finite second moments, we have a martingale property as well as variance estimates and a CLT for subpopulations after one time step.

 \begin{figure}[h!]
\centering
\includegraphics[width=5cm]{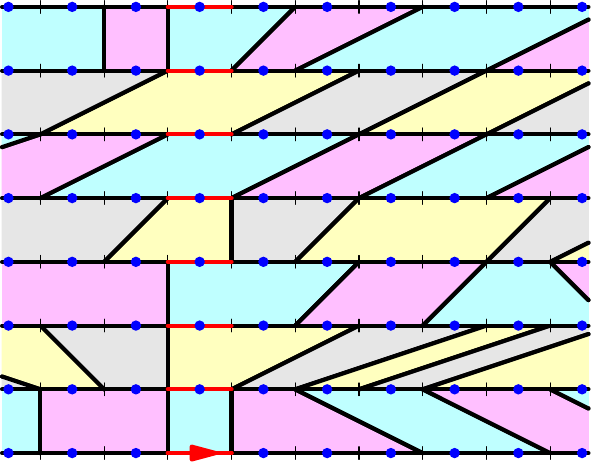} \hfill
\includegraphics[width=5cm]{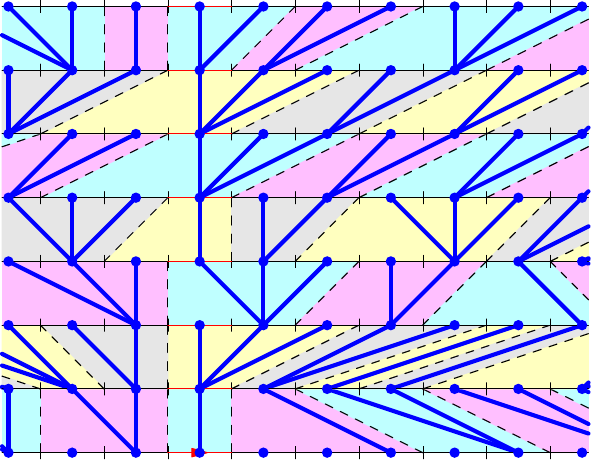} \hfill
\includegraphics[width=5cm]{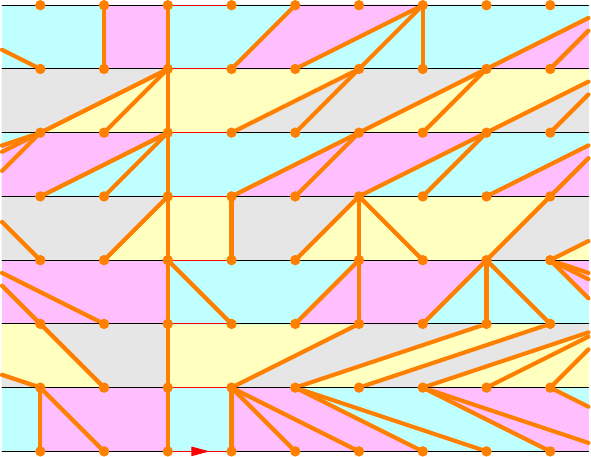}
\caption{Left, a brick wall, with the brick borders in black, and the bricks displayed in various colors. The origin edge $[0,1]\times\{k\}$ of each level $k$ is displayed in red, the brick on its top is size-biased by the variable $B$ whereas all other bricks are sampled iid according to the law $\rho$. The integers points $(i,j)$ are displayed with black crosses and the half-integers points $( i+ \frac{1}{2}, j)$ with blue dots -- except the root node $(\frac{1}{2},0)$, corresponding to the red arrow of the root edge. Center, the associated primal forest $F^{\uparrow}$, in blue. Right, the associated dual forest $F^{\downarrow}$, in orange.} 
\label{fig:bricksall}
\end{figure}

\subsection{Definition of the model}
\label{subsec:defs}

Let $\rho$ be a probability distribution on $\Z_{>0}^2$ such that
\[
Z:=\sum_{i=1}^{+\infty}\sum_{j=1}^{+\infty}i\rho(i,j)<\infty.
\]
We define its \textbf{bottom marginal law} $\beta$ and its \textbf{top marginal law} $\eta$  by 
\[
\beta(i) = \sum_{j \geq 1} \rho(i,j), \ \ \eta(j) =  \sum_{i\geq1} \rho(i,j).
\]
We also define the respective \textbf{bottom-biased} laws for $i,j \geq 1$ by 
\begin{align*}
\rho^*(i,j)=\frac{1}{Z}i\rho(i,j), \qquad 
\beta^*(i)=\frac{1}{Z}i\beta(i),  \quad \mbox{ and }\quad 
\eta^*(j)=\frac{1}{Z}\sum_{i\geq 1}i\rho(i,j).
\end{align*}

\begin{Def}[Brick wall] \label{def:brick}
The \textbf{brick wall with distribution \boldmath$\rho$} is defined as follows. We construct the first row of bricks at height $k=0$, by first drawing a bi-infinite sequence of independent brick dimensions $(B_{i,0},H_{i,0})\sim\rho$ for $i \in \Z \setminus\{0\}$, and $(B_{0,0},H_{0,0})\sim\rho^*$. We think of these bricks as having rigid bottom and top lengths, but flexible vertical sides, so that we need some additional information to place them on the horizontal line. For that purpose, we fix the position of the root brick by requiring that the bottom-left corner of the root brick is $(-U_0,0)$, where $U_0$ is, conditionally on $(B_{i,0},H_{i,0})_{i \in\Z}$, uniform over $\{0,1, \dots, B_{0,0}-1\}$, and that its top-left corner of the root brick is at $(0,1)$. Once the root brick is fixed, we can place all the other bricks inductively, to its left and to its right.

We proceed similarly for the other rows of bricks that we stack on this first one: at each level $j\geq1$, we pick, independently from the lower levels, a new bi-infinite sequence of brick dimensions,  $(B_{i,j},H_{i,j})\sim\rho$ for $i \in \Z \setminus\{0\}$, and $(B_{0,j},H_{0,j})\sim\rho^*$. We then fix the position of the root brick at height $j$ again by requiring that its bottom-left corner is at $(-U_j,j)$, where $U_j$ is uniform over $\{0,1, \dots, B_{0,j}-1\}$, and that its top-left corner is at $(0,j+1)$.
\end{Def}

The horizontal sides of the bricks can be seen as edges so that the brick wall is  a graph with vertex set $\mathbb{Z} \times  \mathbb{Z}_{\geq 0}$ containing in particular all horizontal edges $(i,j) \leftrightarrow (i+1,j)$, see Figure \ref{fig:bricksall}. This graph is rooted at the oriented edge $(0,0) \to (1,0)$. Its information is entirely encoded in the sequence   \begin{eqnarray} \label{eq:defbrick} \big((B_{i,j}, H_{i,j} : i \in \mathbb{Z}), U_j : j \geq 1\big) \quad  \mbox{satisfying} \quad  B_{i,j}, H_{i,j} \geq 1 \quad \mbox{ and } U_j \in \{0,1, ... , B_{0,j}-1\}.  \end{eqnarray}
The reader may have already guessed that we use a size-biased law in our construction to get the following invariance property:

\begin{prop}[Invariance by translation] The law of the brick wall is \textbf{invariant under the shift of the root edge} $(0,0) \to (1,0)$ one unit to the right (or to the left). \label{prop:invariance}
\end{prop}
\begin{proof}
The bottom right corners of the lowest bricks can be seen as a point process on $ \mathbb{Z}$ and since we use the size-biaised distribution for the root brick, a classical result in renewal theory (see e.g. \cite[Proposition 9.18]{kallenberg}) shows that its distribution is the so-called "stationary" renewal process and its law is invariant under shift of the root edge one unit to the right. This shift may cause a relabeling of the vertices at height $1$, but an iterative application of the above shift invariance shows that the law of the entire brick wall is shift invariant. \end{proof}

\begin{Def}[Primal and dual forests] \label{def:forests}Given a brick wall as in \eqref{eq:defbrick} we define two graphs $F^\uparrow$ and $F^\downarrow$ called respectively the \textbf{primal and dual forests}. The primal forest $ F^\uparrow$ has vertex set $ (\mathbb{Z}+ \frac{1}{2}) \times \mathbb{Z}_{\geq 0}$ and its edge set is prescribed as follows:
\begin{itemize}
\item we put a node on the center of each horizontal edge $(i,j) \leftrightarrow (i+1,j)$ of the brick wall (so that these nodes have half-integer abscissae and integer ordinates);
\item in each brick, we put an edge between the right-most bottom node and all the top nodes, see Figure \ref{fig:bricksall}.
\end{itemize}

The dual forest  $F^{\downarrow}$ is defined similarly: its vertex set is $ \mathbb{Z} \times \mathbb{Z}_{\geq 0}$ and its edge set is given by linking, inside each brick, the left-most top vertex with integer coordinates to all bottom vertices with integer coordinates except the right-most one, see Figure \ref{fig:bricksall}.
\end{Def}

Clearly $ F^\uparrow$ is a spanning forest of $ (\mathbb{Z}+ \frac{1}{2}) \times \mathbb{Z}_{\geq 0}$ and for each $i \in \Z$, we denote by $T^{\uparrow}_i$ the tree of $F^{\uparrow}$ that starts at $i+1/2$. By Proposition \ref{prop:invariance}, we have $(T_{i+1}^\uparrow : i \in \mathbb{Z}) \stackrel{(d)}{=} (T_i^\uparrow : i \in \mathbb{Z})$, but once again, the trees $T_i^\uparrow$ are not independent in general. While the trees of $F^{\uparrow}$ and ${F}^{\downarrow}$ are naturally embedded in $\R^2$, we will consider them as intrinsic metric spaces, by endowing them with the graph distance where each edge is of unit length. In order to help the reader distinguish between the primal and the dual forests, the points  $ \in (\mathbb{Z}+ \frac{1}{2}) \times \mathbb{Z}_{\geq 0}$ will be called \textbf{nodes}, whereas those in $\mathbb{Z} \times \mathbb{Z}_{\geq 0}$ will simply be called vertices.

\begin{rem} A particular case of those constructions already appeared in the theory of discrete random planar maps, see Krikun \cite{krikun2005,krikun-uipq} or \cite{curien-legall-modif,curien-menard,carrance-trig-eul,lehericy}.
\end{rem}

\begin{framednote}{Criticality and moment conditions}
In this paper, we only consider \textbf{critical} models, for which the marginal laws $\beta$ and $\eta$ have equal (and finite) expectation:
\begin{equation}
\Expect{B} = \sum_{i \geq 1} i \beta(i)=\Expect{H} = \sum_{i \geq 1}\eta(i)=Z<\infty.
\label{eq:criticality-hyp}
\end{equation}
We also assume that the marginal laws $\beta$ and $\eta$ both have finite third moments:
\begin{equation}
\Expect{H^3}<\infty \text{ and } \Expect{B^3}<\infty,
\label{eq:hyp-third-moment}
\end{equation}
and we exclude the degenerate case where the top and bottom of each brick are equal a.s.:
\begin{equation}
\exists i,j\geq 1\ \ i\neq j \text{ and }\rho(i,j)>0.
\label{eq:hyp-non-degenerate}
\end{equation}
\end{framednote}

\subsection{Special case: Bienaymé--Galton--Watson trees}
\label{subsec:local-cata-bgw}

Consider the special case of a brick distribution where the top and bottom marginals are independent: $\rho = \beta \otimes \eta$. Let $\mu,\nu$ be two probability distributions on $\Z_{>0}$ and $\Z_{\geq 0}$ respectively, such that
\begin{equation*}
\begin{cases}
\eta=\mu \\
\beta \text{ is the law of } B=1+\inf\{i \geq 0 \, \lvert \, \sum_{0 \leq j \leq i} (N_j-1)<0\}, \, N_i \text{ i.i.d.}\sim \nu
\end{cases}
\end{equation*}
Then the evolution of the forest $F^{\uparrow}$ can be described informally as a discrete genealogical process with ``local catastrophes'', where, at each integer time $n \geq 1$:
\begin{itemize}
\item each individual alive at time $n-1$ is independently scheduled to give birth to a number of children according to $\mu$
\item then, each individual alive at time $n-1$ cancels all the impending offspring of a number of individuals of law $\nu$, starting (if the number is positive) with itself, and then going to the individuals on its right (that are not necessarily its siblings).
\end{itemize}
Informally, $B$ is the time between two consecutive moments when the debt of deaths is settled. In the context of continuous branching processes in varying environment, processes with Poissonian ``catastrophes'' have been considered~\cite{bansaye-simatos}, and the term local catastrophes, that we use here and in the title of this paper, is in analogy to those models. Note however that our discretes processes are not in general branching processes, due to the interdependence between the trees. 

\textbf{Bienaymé--Galton--Watson} (hereafter BGW) processes can be seen as special cases of such genealogical process with local catastrophes, where the catastrophes are just the deaths of single individuals. Indeed, for any probability distribution $p$ on $\Z_{\geq 0}$, the genealogical process obtained as above from
\begin{equation*}
\begin{cases}
\mu=(p \, \lvert \, p>0)\\
\nu=\text{Ber}(p_0),
\end{cases}
\end{equation*}
is a BGW process of offspring distribution $p$, started from $\Z$. This can also be directly expressed as a  brick wall model with independent top and bottom distributions, given by
\begin{equation*}
\begin{cases} 
\beta=\text{Geom}(1-p_0)\\
\eta = (p \, \lvert \, p>0).
\end{cases}
\end{equation*}

In this case, the random trees $T_i^\uparrow$ are simply i.i.d.~BGW trees with offspring distribution $p$: interaction between the trees disappear. In the case of a \textit{critical} BGW, if $X\sim p$ has the law of the offspring distribution we have $ \mathbb{E}[X]=1$ and 
$$ \mathbb{E}[H] = \mathbb{E}[B] = \frac{1}{1-p_0}, \quad \mathbb{E}[B^2] = \frac{1+p_0}{(1-p_0)^2}, \quad \mathbb{E}[H^2] = \frac{\mathbb{E}[X^2]}{1-p_0},$$ so that  
$$\sigma^2=\frac{ \mathbb{E}[(H-B)^2]}{ \mathbb{E}[B]} = \mathbb{E}[X^2]-1 = \Var{X}.$$
Let us take the opportunity given by this model to describe the Brownian forest in detail. Recall from \cite{duquesne-legall-mono} the classical coding of a real tree by a positive excursion: if $ \mathbf{e} : [0, \ell] \to \mathbb{R}_+$ is a non-negative continuous function satisfying $ \mathbf{e}(0) = \mathbf{e}(\ell)=0$; it can be seen as the contour function of a pointed real tree $ (\mathcal{T}_{\mathbf{e}}, \varrho)$. We can then consider the pushforward $  \mathbf{n}$ of Itô's excursion measure for linear Brownian motion \cite[Chapter XII]{revuz-yor} on pointed real trees by this coding. In particular $\mathbf{n}$ is a sigma-finite measure on the set $ \mathbb{T}$ of compact pointed real trees (endowed with the pointed Gromov--Hausdorff distance) and we fix the normalization of $ \mathbf{n}$ by requiring that 
\begin{displaymath}
  \mathbf{n}(\{ (\tau, \varrho) : \mathrm{Height}(\tau) \geq x \} ) = \frac{2}{x}, \quad \forall x >0.
\end{displaymath}

For any $ \sigma >0$, if $ \Pi $ is a Poisson point process on $ \mathbb{R}\times \mathbb{T}$ of intensity $ \sigma^{-2} \mathrm{d}t \otimes \mathbf{n}$, then we consider the pointed tree $ \mathcal{F}_\sigma$ obtained for each atom $(t, (\tau, \varrho))$ of $\Pi$ by grafting $(\tau, \varrho)$ through $\varrho$ at the point $t \in \mathbb{R}$. This operation is well-defined in the local pointed Gromov--Hausdorff topology since there is a finite number of trees of height larger than $ \varepsilon$ grafted on each compact set of $ \mathbb{R}$ -- see \cite[Section 6.1.2]{ckm} for a variant of this construction. The random tree $ \mathcal{F}_\sigma$ is pointed at $0 \in   \mathbb{R}$ and has two ends. If $ \underline{F}^\uparrow$ is the infinite random tree obtained by grafting i.i.d.~BGW$(p)$ trees on the line $ \mathbb{Z}$ and pointed at $0$, then we have 
\begin{thm}[Aldous \cite{aldous-crt1}, Le Gall \cite{lg-random-trees}] If $p$ is critical and has finite non zero variance $\sigma^2$, then we have the following convergence for the local pointed Gromov--Hausdorff metric 
$$ \lambda \cdot  \underline{F}^\uparrow \xrightarrow[\lambda\to0]{} \mathcal{F}_\sigma.$$
\label{thm:aldous}
\end{thm}

\subsection{Duality between the primal and dual forests}
\label{subsec:dual}

The terms ``primal'' and ``dual'' forests, as well as their respective constructions, hint at a duality which roughly says that $F^{\downarrow}$ has the law of $F^\uparrow$ for the reverse brick distribution $^t \rho$ of law $(H,B)$ where $(B,H) \sim \rho$. In this section, we formalize this into a precise statement, that will be of use in~\Cref{sec:cv-feller-dual}. To state precise the duality property, we need rather heavy notation, so this section can be skipped at first reading.

\begin{figure}[htp]
\centering
\includegraphics[width=\textwidth]{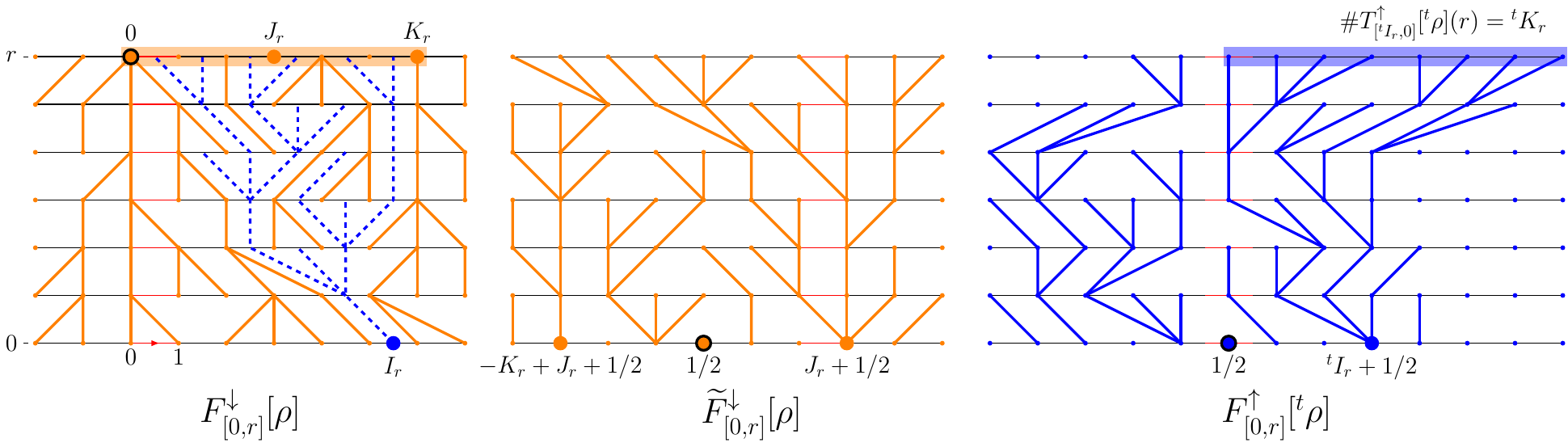}
\caption{The different forests at play in~\Cref{prop: forest duality}: left, the strip ${F}_{[0,r]}^{\downarrow}[\rho]$ of the dual forest associated to $\rho$. Center, $\widetilde{{F}}_{[0,r]}^{\downarrow}[\rho]$, which is the same forest re-rooted at the random index $J_r$, then rotated and shifted. Right, the strip ${F}^{\uparrow}_{[0,r]}[^t\rho]$ of the primal forest associated to $^t\rho$, and, highlighted in light blue, the set of vertices $T^{\uparrow}_{(I_r,0)}[^t\rho](r)$. For the left and right forest, the root edges of the underlying brick wall are depicted in red at each level, and the root vertex is circled in black. Note that the center and rightmost (sub-)forests in these examples are identical as metric spaces, but they have slightly different embeddings in $\R^2$ due to the rooting conventions of our brick walls.}
\label{fig:forest-duality}
\end{figure}

In particular, we make the dependence on the underlying distribution $\rho$ explicit in the notation. Fix $r \geq 1$. For a brick distribution $\rho$, let us denote by ${F}^{\uparrow}_{[0,r]}[\rho]$ (resp. ${F}_{[0,r]}^{\downarrow}[\rho]$) the ``strip'' of the associated primal (resp. dual) forest that lies in $\Z \times [0,r]$. For $i \in \Z$, we write  $T_{(i,0)}^{\uparrow}[\rho]$ for the tree of the primal forest ascending from the node $(i+1/2,0)$, and $T^{\downarrow}_{(i,r)}[\rho]$ the tree of the dual forest descended from the vertex $(i,r)$ (that has at most height $r$ by our definition). For a tree $T$ of height at least $k$, we write $T(k)$ for the set of its vertices at height $k$. Then, let $I_r$ denote the smallest integer $i>0$ such that $T_{(i,0)}^{\uparrow}[\rho]$ has height at least $r$. We define $K_r:=\# T^{\uparrow}_{(I_r,0)}[\rho](r)$. In other words, $K_r$ is the number of descendants at time $r$ of the first primal tree right of the root that survives from time $0$ to time $r$. We can see that $K_r$ is also equal to the first \emph{index} $i\geq 1$ such that the \emph{dual} tree $T^{\downarrow}_{(i,r)}[\rho]$ survives from ordinate $r$ to ordinate $0$. Indeed, a quick look at Figure \ref{fig:forest-duality} should convince the reader that all the dual trees $T^{\downarrow}_{(i,r)}[\rho]$ with $0<i<K_r$ die out before reaching ordinate 0. This observation will be of use in~\Cref{sec:cv-feller-dual}. We also denote by $^tI_r$ and $^tK_r$ the equivalent quantities defined for the reverse brick law $^t\rho$. We can now properly state a duality relation.

\begin{prop}
Let $\widetilde{{F}}_{[0,r]}^{\downarrow}[\rho]$ be the infinite forest obtained by re-rooting ${F}_{[0,r]}^{\downarrow}[\rho]$ at the node $(J_r,r)$, where the index $J_r$ is chosen uniformly at random in $\{1, \dots, K_r\}$, turning it upside-down, and shifting its abscissae by $1/2$ to the right (see~\Cref{fig:forest-duality}). Let $\mu^{\downarrow}_{r}[\rho]$ be the measure on forests $(\dots, T_{-1},T_0,T_1, \dots)$ of height $r$ having density $K_r$ with respect to the law of $\widetilde{{F}}^{\downarrow}_{[0,r]}[\rho]$. On the other hand, let $\mu^{\uparrow}_{r}[^t\rho]$ be the measure on the  the same class of forests, having density $^tK_r$ with respect to the law of ${F}^{\uparrow}_{[0,r]}[^t\rho]$. Then
\begin{equation}
\mu^{\downarrow}_{r}[\rho]=\mu^{\uparrow}_{r}[^t\rho].
\label{eq: forest duality}
\end{equation}
\label{prop: forest duality}
\end{prop}

\vspace{-2em}

\begin{proof}
Even though the trees we are studying are more complicated than the ones appearing in~\cite{curien-legall-modif} (and, in particular, are not independent), we can prove this result similarly to~\cite[Proposition 8]{curien-legall-modif}, with essentially combinatorial arguments.

Indeed, consider a configuration of bricks $\mathcal{C}$ going from height 0 to $r$, with a distinguished root brick at each height level, such that the root brick at height $\ell$ always shares an edge with the root brick at height $\ell+1$ (so that it can be considered as a subset of a brick wall). We denote by $(i_{k,\ell},j_{k,\ell})$ the bottom and top dimensions of the $k$-th brick at height $\ell$ (with the 0-th brick being the root brick, and the negative indices corresponding to the bricks to the left of the root brick). From $\mathcal{C}$, we construct a dual forest ${F}^{\mathcal{C}}=\dots,T^{\mathcal{C}}_{-1},T^{\mathcal{C}}_{0},T^{\mathcal{C}}_{1}, \dots$, where all trees have height less than or equal to $r$. We perform the re-rooting, rotation and shift as above to get a new rooted forest $\widetilde{{F}}^{\mathcal{C}}=\dots,\widetilde{T}^{\mathcal{C}}_{-1},\widetilde{T}^{\mathcal{C}}_{0},\widetilde{T}^{\mathcal{C}}_{1}, \dots$ Let $K^{\mathcal{C}}$ be the first index $i\geq 1$ such that $T^{\mathcal{C}}_{i}(r)\neq \varnothing$, and let $^tK^{\mathcal{C}}$ be the first index $i \geq 1$ such that $\widetilde{T}^{\mathcal{C}}_{i}(r)\neq \varnothing$.

Note that, going from $F^{\mathcal{C}}$ to $\widetilde{F}^{\mathcal{C}}$, we forget the position of the original root vertex at level $0$ in $\mathcal{C}$. However, the dual root tree in $F^{\mathcal{C}}$ becomes in $\widetilde{F}^{\mathcal{C}}$ the first tree $T$ of height $r$ right of the root of  $\widetilde{F}^C$. Thus, the original root vertex of $F^{\mathcal{C}}$ corresponds to one of the nodes at height $r$ in $T$. Since there are precisely $^tK^{\mathcal{C}}$ such nodes, we have
\begin{align*}
\mu_r^{\downarrow}[\rho]\left(\widetilde{{F}}^{\mathcal{C}}\right)&=K^{\mathcal{C}}\cdot{}^tK^{\mathcal{C}}\cdot\frac{1}{K^{\mathcal{C}}}\cdot\frac{1}{i_{0,0}}\frac{1}{Z}i_{0,0}\rho(i_{0,0},j_{0,0})\prod_{k\neq 0}\rho(i_{k,0},j_{k,0})\prod_{1 \leq \ell \leq r-1}\frac{1}{i_{0,\ell}}\frac{1}{Z}i_{0,\ell}\rho(i_{0,\ell},j_{0,\ell})\left(\prod_{k\neq 0}\rho(i_{k,\ell},j_{k,\ell})\right) \\
&=\,^tK^{\mathcal{C}}\cdot\prod_{1 \leq \ell \leq r-1}\frac{1}{Z}\left(\prod_{k}\rho(i_{k,\ell},j_{k,\ell})\right) \\
&=\mu^{\uparrow}_r[^t\rho]\left(\widetilde{{F}}^{\mathcal{C}}\right).
\end{align*} 
\end{proof}

\subsection{Population evolution and martingale property}
\label{subsec:martingales}

We now prove that in a critical brick wall model, the population descended from a set of ancestors in $F^\uparrow$ is a martingale. This is our first step in the understanding of the large scale properties of the trees $T_i^\uparrow$, since it implies in particular that all trees $T_i^\uparrow$ are finite in the critical case.\medskip

Let $a<b \in \Z$. We denote by $M_t([a,b[)$ the number of descendants at time $t$ of the nodes $x$ in $a \leq x < b$ at time $0$, for the primal forest. Necessarily, $M_0([a,b[) = b-a$. We will also sometimes consider populations starting from individuals at time $t_0>0$, which we denote accordingly $M_t([(a,t_0),(b,t_0)[)$, for $t\geq t_0$. Thus, $M_t([a,b[)=M_t([(a,0),(b,0)[)$.

\begin{prop}
\label{prop:slice-of-pop-martingale} If $(H,B) \sim \rho$, we suppose that $ \mathbb{E}[H] = \mathbb{E}[B]$ (criticality) and furthermore that $ \mathbb{E}[B^2], \mathbb{E}[H^2] < \infty$. Then for any $a<b \in \Z$, the sequence $\left(M_t([a,b[) \right)_{t \in \Z_{\geq 0}}$ is a (positive) martingale and a homogeneous Markov chain.
\end{prop}

\begin{proof}
The fact that $\left(M_t([a,b[) \right)_{t \in \Z_{\geq 0}}$ is a Markov chain is immediate from the definition of the model and by invariance by shift of the root edge (Proposition \ref{prop:invariance}). We now want to show that it is a martingale (the filtration $ \mathcal{F}_{n}$ being generated by the first $n$ rows of the brick wall). Let us first check that $ \mathbb{E}[M_1([0,n[)] < \infty$. To see this, we use the notation of~\Cref{def:brick} and the domination 
  \begin{eqnarray}\label{eq:boundtrivial}M_1([0,n[) \leq H_{0,0} + \sum_{i=1}^n H_{i,0}.  \end{eqnarray}
While $H_{i,0}$ for $i \geq 1$ are i.i.d with  a second moment, the first moment of $H_{0,0}$ is also bounded by $ \mathbb{E}[H_{0,0}] =  \frac{1}{ \mathbb{E}[B]}\mathbb{E}[H\cdot B] < \frac{1}{ \mathbb{E}[B]}\mathbb{E}[H^2+ B^2] < \infty$. By the Markov property, it  remains to show that $\mathbb{E}[M_1([a,b[)] = (b-a)$.
By invariance of the brick wall by shift of the root edge we have $\mathbb{E}[M_1([a,b[)] = \mathbb{E}[M_1([0,b-a[)]$, and so by linearity of expectation, there exists $c>0$ so that  for any $n \in \Z_{\geq 0}$,
\begin{eqnarray}
  \Expect{M_1([0,n[)}=c \, n. \label{eq:defc}
\end{eqnarray}
One can prove that $c=1$ by a direct tedious calculation (splitting the calculation according to the value of the last bottom-right corner of a brick before $n$), or using a limiting argument: putting the root brick of dimensions $(B_{0,0}^*,H_{0,0})$ aside, the law of large numbers immediately entails that almost surely we have
$$ \lim_{n \to \infty} \frac{M_1([0,n[)}{n} =  \frac{ \mathbb{E}[H]}{ \mathbb{E}[B]} =1.$$ On the other hand, it follows from \eqref{eq:boundtrivial} that $(\frac{M_1([0,n[)}{n} : n \geq 1)$ is uniformly integrable. By dominated convergence, we deduce that $n^{-1}\cdot \Expect{M_1([0,n[)} \to 1$, hence by \eqref{eq:defc} we must have $c=1$.
\end{proof}

\subsection{Behavior of subpopulations after one time step}
\label{sec:one-time-step}
In this section we establish variance estimates for one step of the Markov chain describing the population evolution which will be useful in~\Cref{sec:fidi} when dealing with the full scaling limit of $(M_k ([0,n[) : k \geq 0)$.
For ease of notation, in this section we will write $M_k^{(n)}$ for $M_k([0,n[)$. We will first obtain an asymptotic estimate for the variance of $M_1^{(n)}$, then prove a CLT for $M_1^{(n)}$.

\paragraph{Variance.}
In the terminology of~\cite{cpr-book-2022}, $M_1^{(n)}$ is a \textbf{inhomogeneous compound renewal process}. The marginal laws $\beta$ and $\eta$, as well as the bottom-biased laws $\beta^*,\eta^*$ all have finite variances -- remember that for the root brick, we take the portion of the bottom that lies to the right of 0, whose law is uniform on $\{1, \dots, b\}$, conditionally on $B^*=b$, and the top of the root brick has the bottom-biased law $\eta^*$. Thus,~\cite[Theorem 1.2.1]{cpr-book-2022} applies and gives the following estimate.
\begin{prop}  \label{prop:varianceonestep}
We have
\begin{equation}
\Var{M_1^{(n)}}=n\sigma^2+\smallo{n},
\label{eq:pop-variance-after-one-step}
\end{equation}
where 
\begin{equation}\sigma^2=\frac{1}{Z}\Expect{(H-B)^2}=\frac{1}{Z}(\sigma^2_{\eta}+\sigma^2_{\beta}-2\textnormal{Cov}(B,H)),
\label{eq:sigma-definition}
\end{equation}
for $(B,H)$ a random variable of law $\rho$, with $Z=\Expect{\beta}=\Expect{\eta}$.
\end{prop}

\begin{rem}
  It is at this point that we really need to assume that the brick law $\rho$ has finite third moments, to be able to get to Brownian scaling limits later on, because of the size-biasing of the root brick. {At first sight, it might seem that we can require only $\Expect{BH^2} <\infty$ instead of $\Expect{H^3} < \infty$, but later on we will make use of the duality property of Proposition~\ref{prop: forest duality}, so that we need $H$ to have a finite third moment as well.}
\end{rem}

\paragraph{CLT after one time step.}

We now prove a CLT, first for $M_1^{(n)}$, then for tuples of subpopulations after one time step.

\begin{prop} \label{prop:cltonestepsingle}
We have the following convergence in distribution:
\begin{equation}
\frac{M_1^{(n)}-n}{\sqrt{n}}\xrightarrow[n \to \infty]{(d)}\mathcal{N}(0,\sigma^2),
\end{equation}
where $\sigma$ is defined by~\eqref{eq:sigma-definition}.
\end{prop}

\begin{proof}
Let us consider a sequence of independent variables $((B_i,H_i))_{i \geq 0}$, where $(B_0,H_0)$ has law $\rho^*$, and $(B_i,H_i)$ has law $\rho$ for $i \geq 1$, as well as the associated sums $\widehat{H}_n=\sum_{0\leq i \leq n}H_i$, $\widehat{B}_n=-U_0+\sum_{0\leq i \leq n}B_i$, where, conditionally on $B_0$, $U_{0}$ is uniform on $\{0, \dots, B_0-1\}$. Then, by the multidimensional FCLT, we have
\begin{equation}
\Bigg(\frac{\widehat{B}_{\lfloor nt \rfloor}-Znt}{\sqrt{n}}, \frac{\widehat{H}_{\lfloor nt \rfloor}-Znt}{\sqrt{n}} \Bigg)_{t \geq 0} \xrightarrow[n \to \infty]{(d)}\Big(W^{(B)}(t),W^{(H)}(t)\Big)_{t \geq 0},
\label{eq:cv-bidim-brownian}
\end{equation}
where $(W^{(B)}(t),W^{(H)}(t))$ is a bidimensional Brownian motion with covariance matrix $C(\rho)$, the covariance matrix of $\rho$. 

By Skorokhod's representation theorem, let us assume that we are working on a probability space for which the convergence~\eqref{eq:cv-bidim-brownian} holds in the almost sure sense rather than in distribution.  Now, we have $M_1^{(n)}=\widehat{H}_{t_n}$, where $t_n=\inf\{k \geq 0 \lvert \widehat{B}_k > n \}-1$. Since we have assumed almost sure convergence, we have, uniformly in $t \geq 0$, 
\[
\widehat{B}_{\lfloor nt \rfloor} = Znt + \sqrt{n}W^{(B)}(t) + \smallo{\sqrt{n}},
\]
so that with high probability $t_n=n/Z+\smallo{n}$, and, using the continuity of $W^{(B)}$,
\begin{equation*}
Zt_n = n -\sqrt{n}W^{(B)}\Big(\frac{1}{Z}\Big) + \smallo{\sqrt{n}},
\end{equation*}
therefore
\begin{align*}
M_1^{(n)}&=\widehat{H}_{t_n}=Z\tau_n+\sqrt{n}W^{(H)}\Big(\frac{t_n}{n}\Big)+\smallo{\sqrt{n}}\\
&=n+\sqrt{n}\Big(W^{(H)}\Big(\frac{1}{Z}\Big)-W^{(B)}\Big(\frac{1}{Z}\Big)\Big)+\smallo{\sqrt{n}}.
\end{align*}
Since $W^{(H)}-W^{(B)}$ is a Brownian motion of variance $\Expect{(H_1-B_1)^2}$, we obtain that $(M_1^{(n)}-n)/\sqrt{n}$ does converge in distribution to a centered Gaussian variable of variance 
\begin{displaymath}
\frac{\Expect{(H_1-B_1)^2}}{Z}=\sigma^2. \qedhere
\end{displaymath}
\end{proof}

The extension of the previous result to several populations is proved similarly:
\begin{prop}
For $p$ a positive integer, fix $s_1< \dots < s_{p+1} \in \R$, and consider the subpopulations $M_{1}(n,i)=M_{1}([\lfloor ns_i\rfloor, \lfloor ns_{i+1}\rfloor[)$, for $1 \leq i \leq p$. Then, we have the following convergence:
\[
\left(\frac{M_1(n,i)-n(s_{i+1}-s_i)}{\sqrt{n}}\right)_{1 \leq i \leq p}\xrightarrow[n \to \infty]{(d)}\mathcal{N}(\vec{0},\sigma^2\cdot\text{Diag}(\{s_{i+1}-s_i\})),
\]
where $\text{Diag}(\{s_{i+1}-s_i\})$ is the diagonal $p \times p$ matrix with coefficients $s_{i+1}-s_i, \, 1 \leq i \leq p$.
\label{prop:clt-tuple-of-pop-slices}
\end{prop}

\paragraph{Covariance.}
As a corollary of the previous results we have the following estimate for the covariance between different subpopulations after one time step.
\begin{prop} \label{prop:bound on covariance}
For any $C_1,C_2,C_3 \in \R$ with $0 <C_1\leq C_2<C_3$, we have
\[
\textnormal{Cov}\left(M_1([0,\lfloor C_1n\rfloor[), M_1([\lfloor C_2n\rfloor, \lfloor C_3n\rfloor[)\right)=o(n) \text{ as } n \to \infty.
\]

\end{prop}

\begin{proof}
For ease of notation, let us assume $C_1=C_2=1$ and $C_3=2$, and let us define 
\[
m(n):=\frac{M_1([0,n[)-n}{\sqrt{n}}, \ \ m'(n):=\frac{M_1([n,2n[)-n}{\sqrt{n}}, \ \ v(n):=m(n)m'(n).
\]
From~\Cref{prop:clt-tuple-of-pop-slices} we have
\[
 m(n)^2  \xrightarrow[n \to \infty]{(d)}\sigma^2\cdot \chi^2_1
\] 
\[
 m'(n)^2  \xrightarrow[n \to \infty]{(d)}\sigma^2\cdot \chi^2_1,
\] 
and furthermore~\eqref{eq:pop-variance-after-one-step} implies convergence of their means. It follows from the equality case in Fatou's lemma (see \cite[Lemma 4.11]{kallenberg}) that $(m(n)^2 : n \geq 1)$ and $(m'(n)^2 : n \geq 1)$ are uniformly integrable. Moreover, we have
\begin{align*}
\Expect{\lvert v(n)\rvert\Indic{v(n)>K}}\leq& \Expect{\lvert m(n) m'(n) \rvert \left( \Indic{m(n)>K}+ \Indic{m'(n)>K}\right)}\\
\leq& \frac{1}{2}\Expect{m(n)^2}\Expect{m'(n)^2\Indic{m'(n)>K}}+\frac{1}{2}\Expect{m(n)^2\Indic{m(n)>K}}\Expect{m'(n)^2},
\end{align*}
so that $v(n)$ is uniformly integrable as well. Using~\Cref{prop:clt-tuple-of-pop-slices} once again, $(m(n),m'(n))$ converges in distribution to $(Y_1,Y_2)\stackrel{(d)}{=}\mathcal{N}(\vec{0},\sigma^2\cdot I_2)$, so that
\[
\frac{1}{n}\textnormal{Cov}\left( M_1([0,n[)M_1([n,2n[)\right)=\Expect{v(n)}\xrightarrow[n\to \infty]{}\Expect{Y_1Y_2}=0.
\]
\end{proof}

\section{Scaling limits of subpopulations}
\label{sec:fidi}

Recall that we write $M_k^{(n)} \equiv M_{k}([0,n[)$ for $k \geq 1$ for ease of notation. From what precedes, we expect that $(n^{-1} \cdot M_{\lfloor n t \rfloor }^{(n)} : t \geq 0)$  converges in distribution to a Feller diffusion process. We will actually prove  a stronger and more general result, Theorem \ref{thm:general-martingale-diffusion-cv-intro-intro}, that applies to our case. Recall the notation from Theorem \ref{thm:general-martingale-diffusion-cv-intro-intro}, in particular of the Markov martingales $(A^{(n)}_{k} : k \geq 0)$ with values in the integers and of the extinction times $\tau^{(n)}_{0}$.

First, we shall prove the scaling limit result \eqref{eq:joint-cv-pop-slice-ext-time} using standard stochastic analysis tools, in Sections~\ref{sec:feller-in-strip} and~\ref{sec:feller-cv-extinction-times}. This is then used in Section~\ref{sec:long-thin-branch-and-kolmo} to prove~\eqref{eq:survival-one-tree} which is more subtle and requires to understand, in the discrete setting, the mechanism leading to $\tn{1}{0}\geq n$.
Finally, in Section~\ref{sec:multiplefellers}, we go back to our forests with local catastrophes and prove a version of~\eqref{eq:joint-cv-pop-slice-ext-time} for tuples of population marginals $(M_{\lfloor nt \rfloor}^{(n)})$, as well as for the dual subpopulations.

\vspace{1em}

Before getting into the details of the proof of Theorem~\ref{thm:general-martingale-diffusion-cv-intro-intro}, let us recall an important result about the Feller diffusion $X$, that we will make use of many times in what follows. Since $X$ can equivalently be seen as a continuous-state branching process (CSBP) starting from $X_{0}=1$ and with branching mechanism $\psi(u)=(\sigma^2/2)u^2$, from classical properties of CSBPs (see for instance~\cite[Chapter 2]{legall-book-zurich}), the law of the extinction time $\theta_{0}$ is explicit:
\begin{eqnarray}
  \Prob{\theta_{0} \leq t}= \mathrm{e}^{-\frac{2}{\sigma^2 t}}. \label{eq:bound-extinction-time-Feller}
\end{eqnarray}

\subsection{Step 1: SDE approximation in a strip}
\label{sec:feller-in-strip}

To prove~\eqref{eq:joint-cv-pop-slice-ext-time}, we first prove a weaker result for the process $(A_k^{(n)})_{k\geq 0}$, up to the time it leaves the strip with horizontal borders at levels $n\epsilon$ and $n/\epsilon$ for fixed $\epsilon>0$ , see~\Cref{fig:cv-in-strip}.

We denote by $\tn{n}{I}$ the hitting time of any set $I\subset \R_+$ for $(A_k^{(n)})_k$, i.e. $\tn{n}{I}=\min\{k \geq 0 |A_k^{(n)}\in I\}$. In this section, for ease of notation, we write  $\tn{n,\epsilon}{}$ for the discrete hitting time $$\tn{n,\epsilon}{} = \tn{n}{]\epsilon n,\frac{n}{\epsilon}[^c}=\min\{k \geq 0 : A_k^{(n)}\in [0,\epsilon n]\cup[n/\epsilon, +\infty[\},$$ and likewise for hitting times for the Feller diffusion process $X$, that is $\tcm{\epsilon}{}=\tc{]\epsilon,\frac{1}{\epsilon}[^c}=\inf\{t \geq 0 : X_t\in [0,\epsilon ]\cup[1/\epsilon, +\infty[\}$.

\begin{figure}[htp]
  \centering
  \includegraphics[scale=1.1]{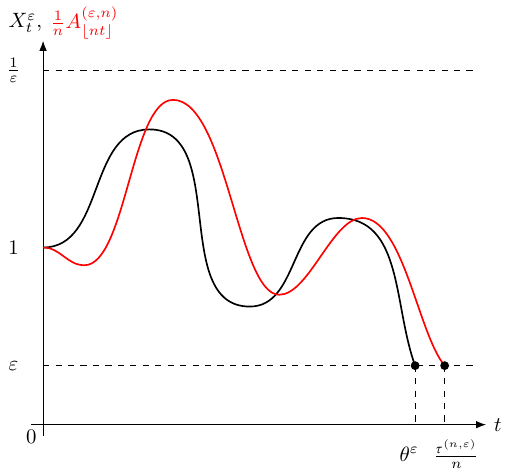}
  \caption{In~\Cref{prop:martingale-diffusion-cv-box}, we prove the convergence of $\frac{A_{\lfloor nt \rfloor}^{(n,\epsilon)}}{n}$, the rescaled discrete process up until it hits $[0,\epsilon]\cup[1/\epsilon, +\infty[$, to $(X_t^{\epsilon})$, the Feller diffusion process up until it hits $[0,\epsilon]\cup[1/\epsilon, +\infty[$.}
  \label{fig:cv-in-strip}
\end{figure}

\begin{prop}
Let $\big(\big(A_k^{(n)}\big)_k\big)_n$ be like in~\Cref{thm:general-martingale-diffusion-cv-intro-intro}. We denote by $\big(A_k^{(n,\epsilon)}\big)_k$ the process $A_k^{(n)}$ stopped at $\tn{n,\epsilon}{}$, that is  $A_k^{(n,\epsilon)}=A_{k\wedge\tn{n,\epsilon}{}}$. Likewise, let $(X_t)$ be as in~\Cref{thm:general-martingale-diffusion-cv-intro-intro}, and define $(X_t^{\epsilon})$ as $(X_t)$ stopped at $\tcm{\epsilon}{}$. Then we have the following convergence in distribution:
\[
\Bigg(\frac{A_{\lfloor nt \rfloor}^{(n,\epsilon)}}{n}\Bigg)_{t \geq 0} \xrightarrow[n \to \infty]{(d)} \big(X_t^{\epsilon}\big)_{t\geq 0}.
\]
\label{prop:martingale-diffusion-cv-box}
\end{prop}

\begin{proof}
We apply to $(A_k^{(n,\epsilon)}/n)_k$ the result~\cite[Corollary 2.2]{ispany-pap} (in our case, since we have a Markov process, the result of Ispány and Pap essentially boils down to a rephrasing of~\cite[Theorem IX.4.21]{jacod-shiryaev}{}, with lighter notation). Following the notation of~\cite{ispany-pap}, we define
\begin{equation*}
U^n_{k+1}=\frac{1}{n}\Big(A^{(n,\epsilon)}_{k+1}-A^{(n,\epsilon)}_{k}\Big), \ \ U^n_0=1, \ \ \mathcal{U}^n_t=\sum_{k=0}^{\lfloor nt \rfloor} U^n_k=\frac{1}{n}A^{(n,\epsilon)}_{\lfloor nt \rfloor},
\end{equation*}
and we take for $\big( \mathcal{F}^n_k \big)_k$ the filtration associated to $\big( A^{(n,\epsilon)}_k \big)_k$. Note that, since $\big(A^{(n)}_k\big)$ is homogeneous, for any $k \in \Z_{\geq 0}$, for any $n_0 \in \{\lceil \epsilon n \rceil, \dots, \lfloor n/\epsilon \rfloor\}$, conditionally on $A^{(n, \epsilon)}_k=n_0$, the variable $A^{(n,\epsilon)}_{k+1}$ has the same law as $A^{(n_0,\epsilon)}_1$. Since $\big(\mathcal{U}^n_t \big)_t$ is a martingale, condition (i) of~\cite[Corollary 2.2]{ispany-pap} is automatically satisfied with $\beta=0$. Moreover, from~\eqref{eq:cond-for-feller-cv1}, we have
\begin{equation*}
\Var{U^n_{k+1} \lvert \mathcal{F}^n_k}\sim\frac{\sigma^2}{n}\mathcal{U}^n_{k},
\end{equation*}
so that condition (ii) of~\cite[Corollary 2.2]{ispany-pap} is satisfied with $\gamma(s,u)=\sigma\sqrt{u}$.

It remains to check condition (iii) of~\cite[Corollary 2.2]{ispany-pap} on big jumps:
\begin{equation}
\sum_{k=1}^{\lfloor nt \rfloor}\Expect{(U^n_k)^2\Indic{\lvert U^n_k \rvert > \kappa} \lvert \mathcal{F}^n_{k-1}}\xrightarrow[n\to \infty]{\mathbb{P}} 0 \quad \forall t,\kappa>0.
\label{eq:ispany-pap-big-jump-condition}
\end{equation}

Now recall from~\eqref{eq:cond-for-feller-cv2} that we have the convergence in distribution
\[
 \frac{\Big(\Delta A^{(n_0,\epsilon)}_1\Big)^2}{n_0}  \xrightarrow[n_0 \to \infty]{(d)}\sigma^2\cdot \chi^2_1,
\] and furthermore \eqref{eq:cond-for-feller-cv1} implies convergence of their means. Like in the proof of~\Cref{prop:bound on covariance}, it follows from \cite[Lemma 4.11]{kallenberg} that the variables in the last display are uniformly integrable and in particular that 
$$ \lim_{K \to \infty} \sup_{n_{0}\geq 1}   \mathbb{E}\left[ \frac{\Big(\Delta A^{(n_0,\epsilon)}_1\Big)^2}{n_0} \mathbf{1}_{(\Delta A^{(n_0,\epsilon)}_1)^{2} > K \cdot n_{0}}\right] =0,$$
so that in particular, for any fixed $\kappa >0$ we have 
$$ \lim_{n_{0} \to \infty} \mathbb{E}\left[ \frac{\Big(\Delta A^{(n_0,\epsilon)}_1\Big)^2}{n_0} \mathbf{1}_{\Delta A^{(n_0,\epsilon)}_1 > \kappa \cdot n_{0} }\right] =0,$$
and this easily implies \eqref{eq:ispany-pap-big-jump-condition}.
\end{proof}

\subsection{Step 2: control on hitting times}
\label{sec:feller-cv-extinction-times}

Let us denote $\widetilde{A}_t^{(n)}=A_{\lfloor nt \rfloor}^{(n)}/n$ for ease of notation. To get the whole convergence~\eqref{eq:joint-cv-pop-slice-ext-time} from~\Cref{prop:martingale-diffusion-cv-box}, we will use Skorokhod's representation theorem, and assume that we are working on a probability space where the convergence of~\Cref{prop:martingale-diffusion-cv-box} occurs almost surely. On this probability space, we have the following result.

\begin{prop}
Let $\mathcal{E}_{\epsilon,n}$ be the event where neither $\widetilde{A}^{(n)}$ nor $X$:
\begin{enumerate}[label=(\roman*)]
\item \label{item:going-high-before-low} Go above level $1/\epsilon$ before going below $\epsilon$;
\item \label{item:going-a-bit-up-after-low} Go back above level $\sqrt{\epsilon}$ after having dropped below level $\epsilon$;
\item \label{item:surviving-long-after-low} Survive for more than $\sqrt{\epsilon}$ unit of time after dropped below level $\epsilon$.
\end{enumerate}
Then, for any $\delta >0$, for $\epsilon$ small enough and $n$ large enough, we have $\Prob{\mathcal{E}_{\epsilon,n}}\geq 1 - \delta$.
\label{prop:good-behavior-outside-box}
\end{prop}

\begin{figure}[htp]
\centering
\includegraphics[scale=1.1]{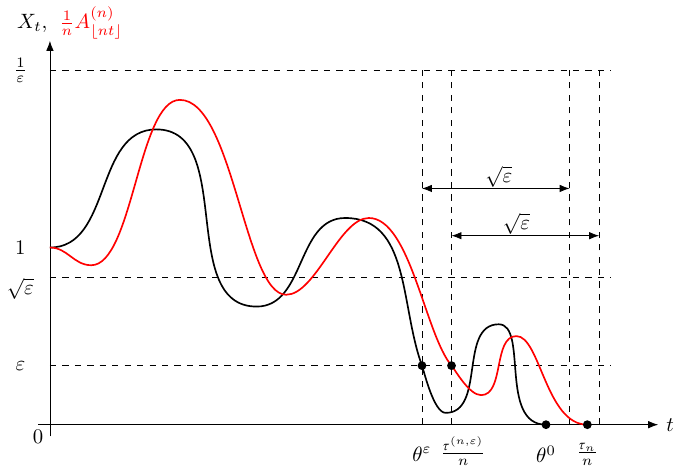}
\caption{A depiction of the event $\mathcal{E}_{\epsilon,n}$ of~\Cref{prop:good-behavior-outside-box}.}
\label{fig:full-cv-event}
\end{figure}

Before getting to the proof of~\Cref{prop:good-behavior-outside-box}, let us explain how it yields~the first part of \Cref{thm:general-martingale-diffusion-cv-intro-intro}. 

\begin{proof}[Proof of \eqref{eq:joint-cv-pop-slice-ext-time} in ~\Cref{thm:general-martingale-diffusion-cv-intro-intro}]  Fix $ \delta,\epsilon>0$. Suppose that $\widetilde{A}^{(n,\epsilon)} \to X^{\epsilon}$ for the topology of uniform convergence on every compact of $ \mathbb{R}_{+}$. Since the extinction time $\tc{0}$ of $X$ is almost surely finite  by \eqref{eq:bound-extinction-time-Feller}, we can choose $ \varepsilon>0$ to be small enough so that $ \mathbb{P}( \tc{0} \leq  \varepsilon^{-1}) \geq 1- \delta$. On this event intersected with the  event $\mathcal{E}_{\epsilon, n}$, we easily deduce that 

$$\tcm{\epsilon}{} \leq \liminf_{n \to \infty} \frac{\tn{n,\epsilon}{}}{n} \leq  \tc{0} \leq \limsup_{n \to \infty} \frac{\tn{n}{0}}{n}$$ 
which combined with 
$$|\tc{0} - \tcm{\epsilon}{}| \leq \sqrt{ \varepsilon} \quad \mbox{ and } \quad \left| \frac{\tn{n,\epsilon}{}}{n} - \frac{\tn{n}{0}}{n} \right| \leq \sqrt{ \varepsilon}$$ 
easily yields that $\limsup_{n \to \infty }|\tau_{0}^{(n)}/n - \theta_{0}| \leq 2 \sqrt{ \varepsilon}$. Similarly, the supremum norm between $\widetilde{A}_t^{(n)}$ and $X$ is upper bounded  by the supremum norm over the time interval $[0, \varepsilon^{-1}]$ plus $2 \sqrt{ \varepsilon}$. Up to decreasing $ \varepsilon>0$, all of this happens on an event of probability at least $ 1- 2 \delta$, \Cref{thm:general-martingale-diffusion-cv-intro-intro} follows. 
\end{proof}

Let us now prove the bounds claimed in~\Cref{prop:good-behavior-outside-box}. Let us start with points~\ref{item:going-high-before-low} and~\ref{item:going-a-bit-up-after-low} for $X$. Since $(X_t)$ is a martingale, and $\tcm{\epsilon}{}$ is an almost surely finite stopping time such that $X_{t\wedge \tcm{\epsilon}{}}\leq 1/\epsilon$ a.s., by the optional stopping theorem we have $\Expect{X_{\tcm{\epsilon}{}}}=1$. Moreover, we have 
\[
\Prob{\{X_{\tcm{\epsilon}{}}=\epsilon\}\cup\{X_{\tcm{\epsilon}{}}=1/\epsilon\}}=1,
\]
so that 
\[
\Prob{X_{\tcm{\epsilon}{}}=\epsilon}=\frac{1-\epsilon}{1-\epsilon^2}.
\]
Point~\ref{item:going-a-bit-up-after-low} follows similarly from Doob's (super)-martingale inequality, since for a nonnegative super-martingale $(S_{t} : t \geq 0)$ with continuous or discrete sample paths, we have
\[
\Prob{\max_{t \geq 0}S_t\geq  A}\leq  \frac{S_{0}}{A}.
\]

For $\widetilde{A}^{(n)}$, we can bound the probabilities of points~\ref{item:going-high-before-low} and~\ref{item:going-a-bit-up-after-low} from the bounds for $X$. Indeed, we can apply to $A^{(n)}$~\Cref{prop:martingale-diffusion-cv-box} with the threshold parameter $\epsilon/2$ (i.e., until the stopping times $\tn{n,\epsilon/2}{}, \tcm{\epsilon/2}{}$), which implies that the probabilities of the events~\ref{item:going-high-before-low} and~\ref{item:going-a-bit-up-after-low} for $\widetilde{A}^{(n)}$ converge to those for $X$.

For point~\ref{item:surviving-long-after-low}, we will use bounds on the extinction times of $\widetilde{A}^{(n)}$ and $X$. For $X$ these are explicitly given in \eqref{eq:bound-extinction-time-Feller}. For $\widetilde{A}^{(n)}$ we will use general bounds for extinction times of super-martingales with polynomial drift proved in  Bertoin, Curien \& Riera \cite[Appendix]{BCR} which we restate here for the reader's convenience:
\begin{lem}[\cite{BCR}]
Let $(S_k : k \geq 0)$ be a $(\mathcal{F}_k : k \geq 0)$-supermartingale with values in $\R_+$ which is absorbed at 0. We suppose that there exist two constants $c>0$ and $\alpha \in \R$ such that we have
\[
\Expect{S_{k+1} - S_k \, | \, \mathcal{F}_k}\leq - c\ S_k^{-\alpha}.
\]
If $\tau =\inf\{k\geq0 : X_k=0\}$ is the absorption time of the chain, then we have for $x >1$
\[
\sup_{x >1}\sup_{n \geq 1}\Prob{\tau \geq xn^{1+\alpha}\, \big|\, S_0=n}\cdot x^{\frac{1}{1+\alpha}} <\infty.
\]
\end{lem}
The desired estimate will follow from the above lemma applied to $(A_k^{(n)})^{1/2}$ (which obviously shares the same extinction time as $(\widetilde{A}_k^{(n)})_k$) and with $\alpha=1$, thanks to the following:
\begin{lem} The sequence $((A_k^{(n)})^{1/2} : k\geq 0)$ is a super-martingale, and there exists $c>0$ such that for all $n \geq 1$ we have 
\begin{equation}
\Expect{(A_{1}^{(n)})^{1/2}-n^{1/2}}\leq \frac{-c}{n^{1/2}}.
\label{eq:bound-on-increment-for-asp-variant}
\end{equation}
\end{lem}

\begin{proof}
The fact that it is a super-martingale follows directly from the conditional Jensen's inequality. Now, to obtain~\eqref{eq:bound-on-increment-for-asp-variant}, observe that
\[
\big(A_{1}^{(n)}\big)^{1/2}=\big(n+\Delta A_{1}^{(n)}\big)^{1/2}=n^{1/2}\Bigg[1+\frac{1}{2}\frac{\Delta A_{1}^{(n)}}{n}-\frac{1}{8}\Big(\frac{\Delta A_{1}^{(n)}}{n}\Big)^2+\smallo{\Big(\frac{\Delta A_{1}^{(n)}}{n}\Big)^2}\Bigg],
\]
so that
\[
\Expect{(A_{1}^{(n)})^{1/2} - n^{1/2}}=\frac{1}{2}\Expect{\frac{\Delta A_{1}^{(n)}}{n^{1/2}}} -\frac{1}{8}\Expect{\frac{(\Delta A_{1}^{(n)})^2}{n^{3/2}}}+\Expect{\smallo{\frac{(\Delta A_{1}^{(n)})^2}{n^{3/2}}}}.
\]
The first term in that sum is zero since $(A_k^{(n)})_k$ is a martingale, and, by~\eqref{eq:cond-for-feller-cv1}, the second behaves, as $n\to\infty$, like
\[
 -\frac{1}{8}\frac{\Expect{(\Delta A_{1}^{(n)})^2}}{n^{3/2}}\sim -\frac{1}{8}\frac{\sigma^2}{n^{1/2}}.
\]
Thus, there exists some integer $M>0$ such that, for $n\geq M$, we have
\[
\Expect{(A_{1}^{(n)})^{1/2} - n^{1/2}}\leq\frac{-\sigma^2}{16n^{1/2}}.
\]
To handle smaller values of $n$, recall that $(A_k^{(n)})^{1/2}$ is a supermartingale:
\[
\Expect{\big(A_{1}^{(n)}\big)^{1/2}}< \Expect{A_{1}^{(n)}}^{1/2}=n^{1/2},
\]
and  this inequality is strict since, by~\eqref{eq:cond-for-feller-cv0}, the martingale is not degenerate. Thus we can indeed find $c >0$ small enough so that the statement of the lemma holds for all $n \geq 1$.
\end{proof}

\subsection{Estimate on extinction time while staying small  and proof of the Kolmogorov-type estimate}
\label{sec:long-thin-branch-and-kolmo}

In order to prove~\eqref{eq:survival-one-tree}, we will use an important corollary of~\eqref{eq:joint-cv-pop-slice-ext-time}, that states that it is unlikely for the process $(A)$ to stay small but positive for a long time. We use the same notation as in \cref{thm:general-martingale-diffusion-cv-intro-intro}:

%%%%%%
%sketch for the long thin branch event
\def\longthinbranch{\tikz[baseline=-0.2ex,scale=0.75]{
\path (0,-1.5) node (A00) {}
      (1,-1.5) node (A01) {}
      (0,-0.75) node (A10) {}
      (1,-0.75) node (A11) {}
      (0,1.5) node (A20) {}
      (1,1.5) node (A21) {}
      (0,2.25) node (A30) {}
      (1,2.25) node (A31) {};

\filldraw[fill=blue!20] {(A00.center)  -- (A01.center)} 
decorate [decoration={snake,segment length=4pt, amplitude=2pt}] {-- (A11.center)} 
{-- (A10.center)}
decorate [decoration={snake,segment length=4pt, amplitude=2pt}] {-- cycle};
\filldraw[fill=blue!20] {(A20.center)  -- (A21.center)} 
decorate [decoration={snake,segment length=4pt, amplitude=2pt}] {-- (A31.center)} 
{-- (A30.center)}
decorate [decoration={snake,segment length=4pt, amplitude=2pt}] {-- cycle};
\draw[ultra thick] (A00.center) -- (A01.center);
\draw[ultra thick] (A10.center) -- (A11.center);
\draw[ultra thick] (A20.center) -- (A21.center);
\draw[ultra thick] (A30.center) -- (A31.center);
\draw[thick, loosely dotted] (0.5,-0.5) -- (0.5,1.25);

\node at (-0.9,-1.2) {$n$};
\node at (-0.9,1.9) {$n$};
\node at (1.7,0) {$K$};
\node at (0.5,-2) {$\leq n$};
\draw[thick,<->] (1.25,-1.5) -- (1.25,2.25);
\draw[thick,<->] (-0.25,-1.5) -- (-0.25,-0.75);
\draw[thick,<->] (-0.25,1.5) -- (-0.25,2.25);
\draw[thick,<->] (0,-1.7) -- (1,-1.7);
}
}
%%%%%%

\begin{cor} There exists a constant $c>0$, such that, for any $K \geq n\geq x$ we have the following bound on the probability that the martingale $A^{(x)}_{k}$ stays positive until time $K$ and stays smaller that $n$ every multiple of $n$ time steps:
  \begin{eqnarray*}
\Prob{\left\{ \begin{array}{l} A^{(x)}_{k\cdot n}\leq n  \\  \forall \, 1 \leq k \leq \frac{K}{n}  \end{array}\right\}\cap \Big\{A^{(x)}_{K}>0 \Big\}}= \P\left(\longthinbranch\right) \leq   \mathrm{e}^{-c K/n}.
  \end{eqnarray*}
\label{lem:bound proba long thin pop}
\end{cor}

\begin{proof} By~\Cref{thm:general-martingale-diffusion-cv-intro-intro}, after evolving during $n$ units of time, the probability of not being extinct starting from an initial value not larger than $n$ is bounded from above by some constant $ \Lambda < 1$ independent of $n \geq 1$. Now, since $(A^{(x)}_{k})_{k}$ is a homogeneous Markov chain, we can combine the above bound with the Markov property  to get that the probability of the event in the corollary is bounded from above by $\Lambda ^{\lfloor K/n \rfloor}$, and this proves the corollary. 
\end{proof} 

We can now prove our Kolmogorov-type estimate.

\begin{proof}[Proof of~\eqref{eq:survival-one-tree} in \Cref{thm:general-martingale-diffusion-cv-intro-intro}] The idea of the proof is the following: We will show that the most likely scenario for the martingale to survive up to time $n$ is to first reach the level $ \varepsilon n$, where $ \varepsilon>0$ is small but fixed. It does so with probability of order $1/( \varepsilon n)$, within  roughly $ \varepsilon n$ units of time and is then of order $ \varepsilon n$. We can then apply the convergence of the hitting times in Theorem \ref{thm:general-martingale-diffusion-cv-intro-intro} to deduce that the probability to stay positive for the $\approx n$ unit of times remaining is of order $  \varepsilon \cdot 2/\sigma^{2}$. This gives a probability $1/( \varepsilon n) \cdot  \varepsilon \cdot 2/\sigma^{2} = \frac{2}{ \sigma^{2}n}$ as desired.

Recall that we denote by $\tn{n}{I}$ the hitting time of any set $I\subset \R_+$ for $(A_k^{(n)})_k$, i.e. 
\[
\tn{n}{I}=\min\{k \geq 0 |A_k^{(n)}\in I\}.
\]
For ease of notation, when $I=[x,+\infty [$ we write $\tau_{\geq x}$ for the hitting time $\tn{1}{[x,+\infty [}=\min\{k \, : \, A^{(1)}_{k}\geq x\}$ and $A$ for the process $A^{{(1)}}$. Notice that by Doob's supermartingale inequality we have 
 \begin{eqnarray} \label{eq:doobsuper} \mathbb{P}( \tau^{(k)}_{ \geq x} < \infty) \leq \frac{k}{x}.  \end{eqnarray}
 The following lemma shows that in a sense this inequality is saturated for $k=1$ and large $x$:

\begin{lem}[On $\tau_{\geq \varepsilon n}$] There exist $c,C>0$ and  a function $\delta(n)$ tending to $0$ as $n\to \infty$, such that for any fixed $ \varepsilon \in (0,1/2)$ and $\zeta >0$ we have as $n \to \infty$
\begin{enumerate}
\item for all $j \geq 1$ we have $ \Prob{ j \varepsilon n \leq \tau_{\geq \varepsilon n} < \infty} \leq \frac{C \mathrm{e}^{- c j} }{ \varepsilon n}$,
\item $ \Prob{A_{\tau_{\geq \varepsilon n}} \geq ( \varepsilon + \zeta) n  \mbox{ and }  \tau_{\geq \varepsilon n} < \infty} =\smallo{\delta(\frac{n\zeta^2}{\epsilon})}$,
\item $ \mathbb{P}( \tau_{\geq \varepsilon n} < \infty) \sim \frac{1}{ \varepsilon n}$ as $n \to \infty$.
\end{enumerate}
\end{lem}

\begin{proof} Fix $ \varepsilon \in (0,1/2)$.
  
\textbf{(1) $\tau_{\geq \varepsilon n} \approx \varepsilon n$ when finite.} We split the event according to whether or not the process has reached above level $n/2$ in the first $(j/2) \varepsilon n$ units of time. Conditioning on the value of the chain at time $jn/2$ and using Corollary~\ref{lem:bound proba long thin pop}, we deduce that up to changing the constant $c>0$
\begin{eqnarray*} \Prob{\tau_{\geq \varepsilon n} \geq j n \mbox{ and } \tau_{0}\geq jn} &\underset{ \mathrm{Cor.} \ref{lem:bound proba long thin pop}}{\leq}& \Prob{\tau_{\geq \varepsilon jn/2} \geq jn/2 \mbox{ and } \tau_{0}\geq jn/2} \cdot  \mathrm{e}^{{-cj/2}}+ \Prob{\tau_{\geq \varepsilon n/2} \leq jn/2} \cdot \mathrm{e}^{{-cj/2}}\\
& {\leq}& \mathrm{e}^{{-cj/2}} \left(\Prob{\tau_{\geq \varepsilon jn/2} \geq jn/2 \mbox{ and } \tau_{0}\geq jn/2} + \Prob{\tau_{\geq \varepsilon n/2} < \infty} \right)\\
& \underset{\eqref{eq:doobsuper}}{\leq}& \mathrm{e}^{{-cj/2}}\left(\Prob{\tau_{\geq \varepsilon jn/2} \geq jn/2 \mbox{ and } \tau_{0}\geq jn/2} + \frac{2}{ \varepsilon n} \right).\end{eqnarray*}
Iterating this inequality for $n, n/2, n/4,...$,  we deduce  that $$\Prob{\tau_{\geq \varepsilon n} \geq jn \mbox{ and } \tau_{0}\geq jn} \leq  \frac{C \mathrm{e}^{- c j} }{ \varepsilon n},$$ for some $C,c>0$ as desired.

\textbf{(2) $A_{\tau_{\geq \varepsilon n}} = \varepsilon n + o_{\mathbb{P}}( n)$ when $\tau_{\geq \varepsilon n}$ is finite.} We now want to bound the probability that the martingale reaches $[ \varepsilon n,+\infty[$ by jumping to a value higher or equal to $( \varepsilon + \zeta)n$ for some $ \zeta >0$. That is we estimate
\[
\Prob{A_{\tau_{\geq \epsilon n}}\geq ( \varepsilon + \zeta)n},
\]
with the convention that if $\tau_{\geq \varepsilon n}= \infty$ then $A_{\tau_{\geq \epsilon n}}=0$. To do so, we decompose this event according to the highest dyadic level $\epsilon n/2^k$ reached before crossing level $ \varepsilon n$: for $k \geq 1$ consider the event 
$$ \Prob{A_{\tau_{\geq \epsilon n}}\geq (\epsilon + \zeta) n \mbox{ and } \sup_{x< \tau_{\geq \epsilon n}}A_x \in  \left[\frac{\epsilon n}{2^{k}} , \frac{\epsilon n}{2^{k-1}} \right[ }.$$
Applying the Markov property at the first time when the process belongs to $[\frac{\epsilon n}{2^{k}}, \frac{\epsilon n}{2^{k-1}}[$, by Doob's inequality \eqref{eq:doobsuper}, this probability is upper bounded by 
$$  \frac{2^{k}}{ \varepsilon n} \cdot \sup_{\frac{\epsilon n}{2^{k}} \leq x < \frac{\epsilon n}{2^{k-1}}} \Prob{A^{(x)}_{\tau^{(x)}_{\geq \epsilon n}}\geq (\epsilon + \zeta) n \mbox{ and } \sup_{y< \tau^{(x)}_{\geq \varepsilon n}}A^{(x)}_{y} \leq  \frac{ \varepsilon n}{2^{{k-1}}}}.$$
We now further split according to the scale of $\tau^{(x)}_{\geq \varepsilon n}$ in the second probability, and estimate for $j \geq 0$
 \begin{eqnarray} \label{eq:borne1} \sup_{\frac{\epsilon n}{2^{k}} \leq x < \frac{\epsilon n}{2^{k-1}}}\Prob{A^{(x)}_{\tau^{(x)}_{\geq \epsilon n}}\geq (\epsilon + \zeta) n \mbox{ and } \sup_{y< \tau^{(x)}_{\geq \varepsilon n}}A^{(x)}_y \leq  \frac{ \varepsilon n}{2^{{k-1}}} \mbox{ and }  \tau^{(x)}_{\geq \varepsilon n} \in [j2^{{-k}}\varepsilon n, (j+1) 2^{{-k}}\varepsilon n[}.  \end{eqnarray}
 On this event, the process first has to stay in the strip $[0, 2^{-(k-1)}\varepsilon n]$ for $j 2^{-k}\varepsilon n$ units of time, which by Corollary~\ref{lem:bound proba long thin pop} costs $ \mathrm{e}^{-c j }$, and then perform a jump of size at least $\zeta  n$ somewhere in the $ 2^{-k}\varepsilon n$ units of time in the interval $[j 2^{-k}\varepsilon n, (j+1) 2^{-k}\varepsilon n]$. It is thus bounded above by 
$$ \eqref{eq:borne1} \leq \mathrm{e}^{-c j }  \cdot 2^{-k}\varepsilon n \cdot \sup_{\frac{\epsilon n}{2^{k}} \leq x < \frac{\epsilon n}{2^{k-1}}} \Prob{\Delta A^{(x)}_{1} \geq  \zeta  n}.$$
Similarly to the proof of Proposition~\ref{prop:bound on covariance}, the hypothesis made on $A$ shows that $(\Delta A^{{(x)}}_{1})^{2} / x$ is uniformly integrable as $x \to \infty$, in particular we have 
$$ \mathbb{P}( \Delta A^{{(x)}}_{1} \geq k) \leq \frac{x}{k^{2}}\cdot \delta (k^{2}/x),$$ where $\delta$ is a bounded function that tends to $0$ at $\infty$.  Up to multiplying $\delta$ by some constant, we then have 
$$\eqref{eq:borne1} \leq 	\mathrm{e}^{-c j }  \cdot 2^{-k}\varepsilon n \cdot \frac{2^{-(k-1)} \varepsilon n}{( \zeta n)^{2}} \cdot \delta ( n \zeta^2  2^{k}/ \varepsilon).$$
Since $\delta(x)$ is bounded and $\delta(x)\xrightarrow[x \to \infty]{}0$, summing the previous bound over $j, k\geq0$, we get, for any fixed $\epsilon, \zeta, n$,
\begin{equation}
\Prob{A_{\tau_{\geq \epsilon n}}\geq (\epsilon + \zeta) n}\leq \frac{\epsilon}{\zeta^2n}\widetilde{\delta}\left(\frac{\zeta^2n}{\epsilon}\right),
\label{eq:expecsummable}
\end{equation}
where $\widetilde{\delta}$ is also a bounded function going to 0 at infinity.
  
  \textbf{ (3) Saturating Doob's inequality.} For ease of notation, let us write $\tau$ for $\tau_{0} \wedge \tau_{\geq \varepsilon n}$. The upper bound is given by \eqref{eq:doobsuper}. For the lower bound, notice that as long as $k < \tau$ by our assumptions we have
\[
\Expect{A_{k+1}-A_k\, | \, \mathcal{F}_k}^2\leq \Expect{(A_{k+1}-A_k)^2\, | \, \mathcal{F}_k}\leq C A_k\leq C\epsilon n.
\]
and, on the other hand, recall from Corollary  \ref{lem:bound proba long thin pop} that $\tau$ is of finite expectation (which however may diverge with $n$). It follows that the martingale $A$ stopped at $\tau$ is bounded in $L^{2}$ and by the optional stopping theorem we have 
\begin{align*}
1=\Expect{A_{\tau}}  =\Expect{A_{\tau }\Indic{A_{\tau}\geq (\epsilon  + \zeta)n}}+\Expect{A_{\tau}\Indic{A_{\tau}\in [\epsilon n, (\epsilon +\zeta) n)}}
\end{align*}
The second term in the right-hand side is bounded above by  $ (\varepsilon+ \zeta) n\mathbb{P}( \tau = \tau_{\geq \varepsilon n})$ while  the first term is negligible in front of $1$ thanks to \eqref{eq:expecsummable}: when $  \varepsilon, \zeta>0$ are fixed we have as $n \to \infty$
$$ \Expect{A_{\tau }\Indic{A_{\tau}\geq (\epsilon  + \zeta)n}} \leq 4 n \sum_{y=0}^\infty  \Prob{A_{\tau } \geq ( \varepsilon + \zeta +y) n} \underset{\eqref{eq:expecsummable}}{=} o\left( 4n \sum_{y=0}^\infty \frac{ \varepsilon}{ (\zeta+y)^2 n}\right) = o(1).$$
For fixed $ \varepsilon, \zeta>0$, as $n \to \infty$, we deduce that the event $\tau_{\geq \varepsilon n}< \tau_0$ happens with probability of order at least  $1/ ( (\varepsilon +\zeta) n)$. Letting $\zeta \to 0$, we deduce that $ \mathbb{P}(\tau_{\geq \varepsilon n}< \tau_0)$ is indeed asymptotic to $1/( \varepsilon n)$ and, that conditionally on that event, the process $A$ at that time is roughly equal to $ \varepsilon n$.
\end{proof}

Equipped with the above lemma, the proof of~\eqref{eq:survival-one-tree} in \Cref{thm:general-martingale-diffusion-cv-intro-intro} is clear: for fixed $ \varepsilon>0$, as $n \to \infty$, the event $\tau_{\geq \varepsilon n}< \tau_0$ happens with probability of order $1/ ( \varepsilon n)$ and, conditionally on it, the process $A$ at that time is roughly equal to $ \varepsilon n$. On this event, and conditioning on the past before $\tau_{\geq \varepsilon n}$,  by the convergence to a Feller diffusion and~\eqref{eq:bound-extinction-time-Feller}, we deduce that the process may then survive for $n$ steps with probability of order $2 \varepsilon /\sigma^2$, which gives the lower bound on the Kolmogorov estimate. 

For the upper bound, by the previous lemma, we can focus on the event $\{\tau_{\geq \varepsilon n}<  \sqrt{ \varepsilon} \cdot n, A_{\tau_{\geq \varepsilon n}} \in ( \varepsilon, (1 + \varepsilon) \varepsilon n)\}$ since the complementary event has probability small in front of $1/n$ when $n \to \infty$ for a fixed $ \varepsilon>0$. By the above argument, the probability of this event is roughly $1/ (\varepsilon n)$ and the process then has to survive at least for $ n-  \sqrt{ \varepsilon}n$ units of time, which by  \eqref{eq:bound-extinction-time-Feller} has probability of order $ \varepsilon \cdot 2/(\sigma^{2}(1- \sqrt{ \varepsilon}))$. \end{proof}

\subsection{Convergence of tuples of subpopulations in the brick wall model}
\label{sec:multiplefellers}

We now go back to our brick wall model, and prove extensions of~\Cref{thm:general-martingale-diffusion-cv-intro-intro}, first for tuples of subpopulations $(M_{t}([a, b[))$, then for tuples of \emph{dual} subpopulations, using the duality of Proposition~\ref{prop: forest duality}.

\subsubsection{Convergence of several populations  to independent Feller diffusions}

We now prove an extension of~\Cref{thm:general-martingale-diffusion-cv-intro-intro} in the case of several subpopulations $(M_{t}([a, b[))$. Indeed, even if the discrete subpopulations are not independent themselves, their scaling limit is given by independent Feller diffusions.

\begin{prop}
For $p \geq 1$ a positive integer, fix $s_1< \dots < s_{p+1} \in \R$, and consider the subpopulations $M_{\lfloor nt \rfloor}^{(n,i)}=M_{\lfloor nt\rfloor}([\lfloor ns_i\rfloor, \lfloor ns_{i+1}\rfloor[)$, for $1 \leq i \leq p$ and $0 \leq t \leq 1$, along with their respective extinction times $\tau^{(n,i)}$. Then, we have the following joint convergence:
\[
\Bigg(\Bigg(\frac{M_{\lfloor nt \rfloor}^{(n,i)}}{n} \Bigg)_{t\geq 0},\frac{\tau^{(n,i)}}{n}\Bigg)_{1 \leq i \leq p} \xrightarrow[n \to \infty]{(d)} \Big(\big(X_t^{(i)}\big)_{t\geq 0},\theta^{(i)}\Big)_{1\leq i \leq p},
\]
where, the $(X^{(i)})$ are independent Feller diffusion processes with respective initial conditions $X_0^{(i)}=s_{i+1}-s_i$, variances $(s_{i+1}-s_i)\sigma^2$, and extinction times $\theta^{(i)}$.
\label{coro:cv-of-tuple-of-pop-slices}
\end{prop}

\begin{proof}
We will only detail the proof for $p=2$ and $s_1=0,s_2=1,s_3=2$, as the general case can be obtained similarly but needs much heavier notation. In that case, we are interested in $(M_{\lfloor nt \rfloor}^{(n,1)},M_{\lfloor nt \rfloor}^{(n,2)})=(M_{\lfloor nt\rfloor}([0,n[),M_{\lfloor nt\rfloor}([n,2n[))$, and their extinction times $\tau^{(n,1)},\tau^{(n,2)}$.

Note that $(M_{\lfloor nt \rfloor}^{(n,1)},M_{\lfloor nt \rfloor}^{(n,2)})$ is a bivariate martingale and a Markov chain. We will thus apply the same strategy as for a single subpopulation. We first define $(M_{\lfloor nt \rfloor}^{(n,1,\epsilon)},M_{\lfloor nt \rfloor}^{(n,2,\epsilon)})$ as the couple of subpopulations, up to the first time one of them leaves the strip $(n\epsilon, n/\epsilon)$. Then, we apply~\cite[Corollary 2.2]{ispany-pap}, this time to the bivariate sequence $(U_k^{(n,1)},U_k^{(n,2)})$, defined as
\[
U_{k+1}^{(n,i)}=\frac{1}{n}\left(M_{k+1}^{(n,i,\epsilon)}-M_{k}^{(n,i,\epsilon)}\right), \ \ U^{(n,i)}_0=1, \ \ \mathcal{U}^{(n,i)}_t=\sum_{k=0}^{\lfloor nt \rfloor} U^{(n,i)}_k=\frac{1}{n}M^{(n,i,\epsilon)}_{\lfloor nt \rfloor}.
\]
Condition (i) of~\cite[Corollary 2.2]{ispany-pap} is once again satisfied with $\beta=0$ since $(\mathcal{U}^{n,i}_t)_t$ is a martingale. Moreover, from~\eqref{eq:pop-variance-after-one-step} and~\Cref{prop:bound on covariance}, condition (ii) is also satisfied with 
\[
\gamma(s,u)_{i,j}=\sigma\sqrt{u_i}\cdot\delta_{i,j}.
\]
To check condition (iii), we can also proceed similarly to the case of a single subpopulation. Indeed, from~\Cref{prop:clt-tuple-of-pop-slices}, $(M_1^{(n,1)},M_1^{(n,2)})$, rescaled by $\sqrt{n}$, converges to a couple of independent Gaussian variables, so that 
\[
\frac{1}{n}\lVert \Delta M_1^{n} \rVert^2=\frac{1}{n}\left((M_1^{(n,1)}-n)^2+(M_1^{(n,2)}-n)^2\right)\xrightarrow[n \to \infty]{(d)}\sigma^2\chi_2^2
\]
and, from~\eqref{eq:pop-variance-after-one-step} we have convergence of their means
\[
\Expect{\frac{1}{n}\lVert \Delta M_1^{n} \rVert^2}\xrightarrow[n \to \infty]{}2\sigma^2=\Expect{\sigma^2\chi_2^2}.
\]
Using once again \cite[Lemma 4.11]{kallenberg}, we can show that this implies condition (iii) of~\cite[Corollary 2.2]{ispany-pap}.

We can then conclude with the same chain of arguments as for~\Cref{thm:general-martingale-diffusion-cv-intro-intro}, observing that the tightness of the rescaled extinction time $\frac{\tau_{2n}}{2n}$ obtained above, implies the joint tightness of $(\frac{\tau^{(n,1)}}{n},\frac{\tau^{(n,2)}}{n})$, since we have $0 \leq \tau^{(n,i)} \leq \tau_{2n}$.
\end{proof}

A better way to phrase the above theorem is to use the \textbf{Feller flow}. This flow is constructed in \cite{dawson-li} (see also \cite{bertoinlegall2} for a related construction) and gives a two-parameter process $(X_{t}(x) : t \geq 0, x \geq 0)$ which satisfies the following conditions:
\begin{itemize}
\item For each $x_{1}<x_{2}< ...<x_{k} $, the processes $(X_{t}(x_{i+1})-X_{t}(x_{i}))_{t \geq 0}$ for $1 \leq i \leq k-1$ are c\`adl\`ag  independent  Feller diffusions starting from $x_{i+1}-x_{i}$,
\item For each $t \geq 0$, the process $x \mapsto X_{t}(x)$ is c\`adl\`ag non-decreasing.
\end{itemize}

\subsubsection{Control of the dual forests}
\label{sec:cv-feller-dual}
In the next section we shall use the results established above, but for the dual forest of descending trees. To do so, we will of course use the duality property~\eqref{eq: forest duality}, and the following control on the Radon--Nikodym derivatives involved. Recall in particular from Section \ref{subsec:dual}, that $I_{r}$ is the index of the first ascending tree having offspring at level $r$ and that $K_{r}=\#T^{\uparrow}_{(I_r,0)}[\rho](r)$ is the total size of that offspring.

\begin{prop}
\label{prop: absolute-continuity-dual}
We have the following convergence in distribution:
\[
\frac{K_n}{n}\xrightarrow[n \to \infty]{(d)} \mathcal{E}_{2/\sigma^{2}}  \quad \mbox{ and }\quad \frac{I_n}{n}\xrightarrow[n \to \infty]{(d)} \mathcal{E}_{2/\sigma^{2}}
\]
where $\mathcal{E}_{2/\sigma^{2}}$ is an exponential variable of parameter $2/\sigma^2$. Furthermore $(\frac{K_n}{n} : n \geq 1)$ is uniformly integrable.

Recalling notation from Proposition \ref{prop: forest duality}, we deduce that the law of the dual forest strip $n^{-1}\widetilde{{F}}_{[0,n]}^{\downarrow}[\rho]$ is contiguous with respect to the law of the primal strip of the forest with reverse brick distribution, $n^{-1}{{F}}_{[0,n]}^{\uparrow}[^t\rho]$.
\end{prop}

\begin{proof}[Proof of Proposition~\ref{prop: absolute-continuity-dual}] Thanks to Theorem~\ref{thm:general-martingale-diffusion-cv-intro-intro}, we can construct a coupling between the discrete processes and Feller flow (see the end of the previous section), such that we have jointly, for any $x \in \Q$
\[
\frac{1}{n}M_n([0,xn[)\xrightarrow[n \to \infty]{a.s.}X_1(x).
\]
In particular, it is well-known that
$$ \xi = \inf\{ s \geq 0 : X_1(s)>0\},$$
is exponentially distributed and that $X_{1}(\xi)$ is also exponentially distributed. The convergence of $I_{n}/n$ towards $\xi$ follows then from the continuous mapping theorem, but not that of $K_{n}/n$: indeed, it might be the case that several trees reach level $n$ around the position $I_{n}$ and that in the scaling limit those different discrete trees collapse into a single one. However, if this phenomenon happens, the discrete number of trees reaching level $n$ is strictly larger than the continuous one. The crux is that the Kolmogorov estimate \eqref{thm:general-martingale-diffusion-cv-intro-intro} will imply that both variables have the same expectation so this is excluded. Let us elaborate. The convergence towards the Feller flow implies that if $I^{{(1)}}_{n}, I^{{(2)}}_{n},...$ are the indices of the discrete ascending trees that have offspring at level $n$ then we have the following convergence 
$$\{ n^{{-1}}I^{{(i)}}_n, i \geq 0\} \xrightarrow[n\to\infty]{a.s.} \mathcal{S},$$ 
for the local Hausdorff distance on closed sets, and where $ \mathcal{S}$ are the set of points where $X_{1}(\cdot)$ has discontinuities (i.e., such that there exists a continuous tree reaching level $1$ planted at $x$). As recalled above, the random set $ \mathcal{S} = \{ S_{1}<S_{2}< \cdots \}$ is Poisson with intensity $  \frac{2}{\sigma^{2}}$. This convergence does not guarantee that $n^{{-1}}I^{{(i)}}_{n} \to S_{i}$ almost surely since we may have e.g. $n^{{-1}}I^{{(1)}}_{n} \to S_{1}$ and $n^{{-1}}I^{{(2)}}_{n} \to S_{1}$, but we have at least  $\limsup n^{{-1}}I^{{(i)}}_{n} \leq S_{i}$. However, from~\eqref{eq:survival-one-tree}, we have
\[
\Expect{\# \{ i \geq 1 : I^{{(i)}}_{n} \leq A n\} }= An\cdot\Prob{M_n([0,1[)>0}\xrightarrow[n \to \infty]{}\frac{2A}{\sigma^2}=\Expect{ \mathcal{S} \cap [0,A]},  \ \ 
\]
and so we must have the corresponding lower bound $\limsup n^{{-1}}I^{{(i)}}_{n} \geq S_{i}$. In particular, $K_{n}/n$ converges towards $X_{1}(\xi)$ which is also exponentially distributed (Yaglom's limit law) as desired. The uniform integrability follows from Scheffé's lemma since $K_{n}/n$ as asymptotically the same expectation as its limit law:
$$ \mathbb{E}[K_{n}] = \mathbb{E}[M_{n}([0,1[) \mid M_{n}([0,1[) >0] \underset{\eqref{eq:survival-one-tree}}{\sim} \frac{1}{ \frac{2}{\sigma^{2} n}} \sim n \mathbb{E}[ \mathcal{E}_{2/\sigma^{2}}].$$

To establish the contiguity, let $(A_{n})_{n \geq 1}$ be a sequence of events such that $ \mathbb{P}( n^{-1}{{F}}_{[0,n]}^{\uparrow}[^{t}\rho] \in A_{n}) \to 0$. Recall notation from and around Proposition \ref{prop: forest duality}, in particular $ \# T^{\uparrow}_{(I_r,0)}[\rho](r) = K_{r}$ where $I_{r}$ is the index of the first ascending tree on the right reaching level $r$. Similarly we set $J_{r}$ for the index of the first tree on the left reaching level $r$ and introduce the similar variables $^{t}K_{r}, ^{t}I_{r}, ^{t}J_{r}$ for the reverse brick distribution $^{t} \rho$. We can then write 
  \begin{eqnarray*} \mathbb{P}\left( n^{-1}\widetilde{{F}}_{[0,n]}^{\downarrow} \in A_{n}\right) & \leq &   \mathbb{P}(K_{n} \leq \varepsilon n) + \mathbb{E}\left[ \left(\mathbf{1}_{n^{-1}\widetilde{{F}}_{[0,n]}^{\downarrow}[\rho] \in A_{n}} \cdot \frac{n}{K_{n}}  \mathbf{1}_{K_{n} \geq \varepsilon n}\right) \frac{K_{n}}{n}\right]\\
  & \underset{ \mathrm{Prop.} \ref{prop: forest duality}}{=}&\mathbb{P}(K_{n} \leq \varepsilon n) + \mathbb{E}\left[ \left(\mathbf{1}_{n^{-1}{F}_{[0,n]}^{\uparrow}[^{t}\rho] \in A_{n}} \cdot  \frac{n}{ ^{t}I_{r} + ^{t}J_{r}} \mathbf{1}_{^{t} I_{n} + ^{t} J_{n} \geq \varepsilon n  }\right)\frac{^{t}K_{n}}{n} \right]  \\
   &\leq & \mathbb{P}(K_{n} \leq \varepsilon n) +  \frac{1}{ \varepsilon} \mathbb{E}\left[ \mathbf{1}_{n^{-1}{F}_{[0,n]}^{\uparrow}[^{t}\rho] \in A_{n}} \frac{^{t}K_{n}}{n} \right].  \end{eqnarray*}
By the previous point applied to the brick model with distribution $\rho$, the first term can be made arbitrary small provided that $ \varepsilon$ is small, then the second term tends to $0$ by the previous point applied to the brick model with distribution $^{t}\rho$.
\end{proof}

\section{Scaling limit of the trees}
\label{sec:forest}

The goal of this section is to obtain the convergence of the whole rescaled genealogical forest to the Brownian forest, for the local pointed Gromov--Hausdorff topology. Recall from Definition~\ref{def:forests} the definition of the trees $T^{\uparrow}_{i}$, and consider the pointed tree $ \underline{F}^{\uparrow}$ obtained by attaching those trees on the bi-infinite line $\R$ pointed at $0$. For our discussion, it will be convenient to use the labeling of the nodes of $\underline{F}^{\uparrow}$ by $\{ (i,j) : i \in \mathbb{Z}, j \geq 0\}$.

Before getting to the proof of Theorem~\ref{thm:main}, let us briefly discuss the local Gromov--Hausdorff topology (for more details see  Chapter 8.1 in \cite{burago-burago-ivanov}, the reference \cite{abraham-delmas-hoscheit} or Section 2 of \cite{curien-legall-brownian-plane}). First recall that a metric space $(E,d)$
is called a length space if, for every $a,b\in E$, the distance $d(a,b)$ coincides with the infimum
of the lengths of continuous curves connecting $a$ to $b$. We also say that $(E,d)$ is boundedly compact if all its closed balls are compact.
 For $r>0$, and $(E,d,\rho)$ a pointed metric space, let us denote by $\mathrm{Ball}_r(E)$ the closed ball of radius $r$ centered at $\rho$ in $E$. A sequence $(E_n,d_n,\rho_n)$ of pointed boundedly compact length spaces is said to converge to $(E,d,\rho)$ in the local Gromov-Hausdorff topology if, for every $r>0$, $\mathrm{Ball}_r(E_n)$ converges to $\mathrm{Ball}_r(E)$, for the
usual Gromov-Hausdorff distance between pointed compact metric spaces. The space of all pointed boundedly compact length spaces (modulo isometries) can be equipped with a metric $\dgh$ which is compatible with the preceding notion of convergence, and is then separable and complete for this metric.

\subsection{High-level proof of Theorem~\ref{thm:main}}
 
\begin{figure}[!h]
  \begin{center}
  \includegraphics[width=13cm]{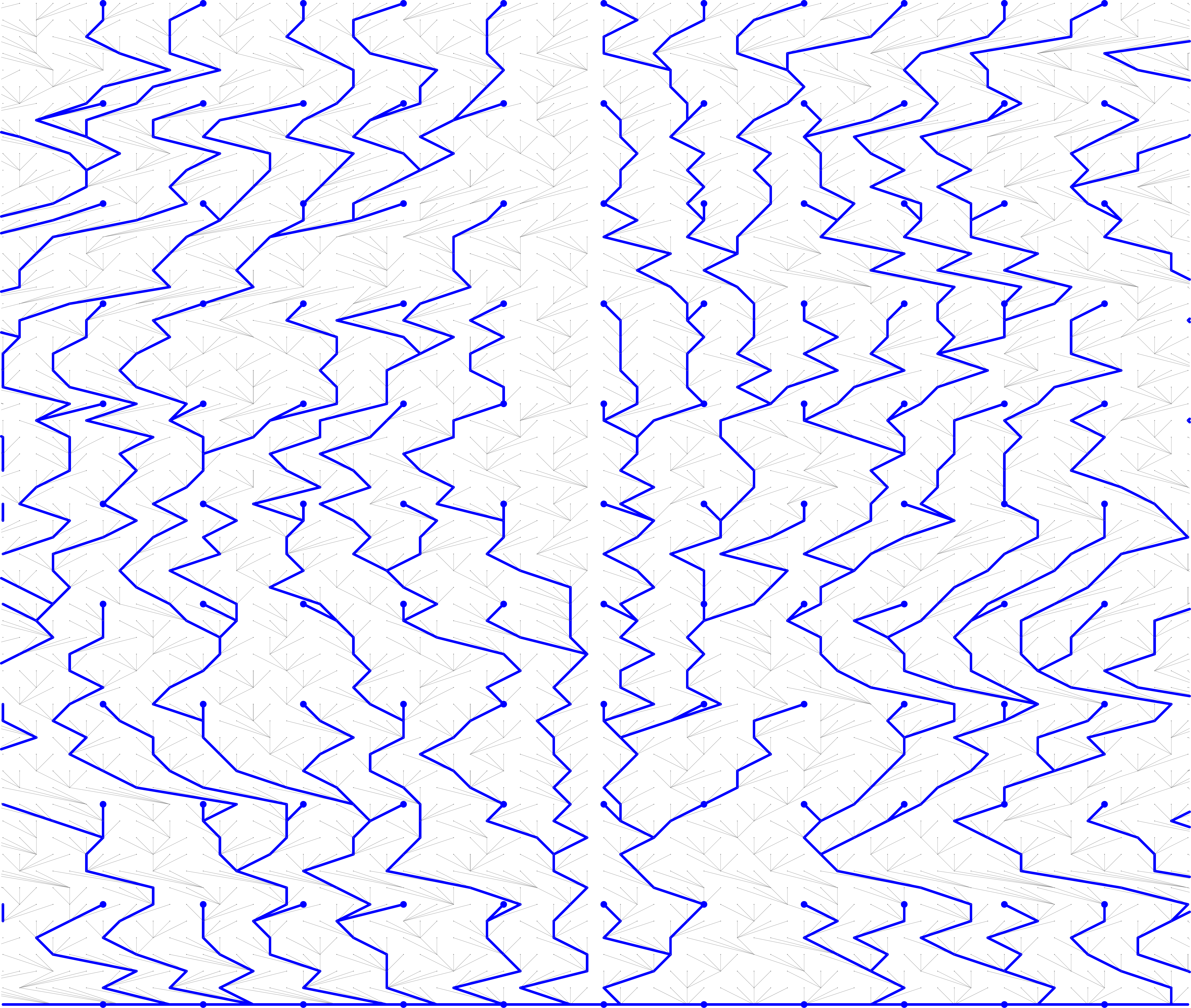}
  \caption{Illustration of the $ \varepsilon$-meshing approximation $\underline{F}^{\uparrow, ( \varepsilon n)}$: the whole forest $\underline{F}^{\uparrow}$ is in grey, and the meshing in blue.}
\label{fig:meshing}
  \end{center}
  \end{figure}

To prove Theorem~\ref{thm:main}, we will make use of the following \emph{meshing} of the forest $\underline{F}^{\uparrow}$. Fix $ \varepsilon>0$ and suppose that $ \varepsilon n$ is an even integer to simplify our notation. An obvious proxy for  $\underline{F}^{\uparrow}$ is obtained by the subforest spanned by the nodes with coordinates $(i \varepsilon n+ \frac{1}{2}, j  \varepsilon n)$ for  $i \in \mathbb{Z}, j \geq 0$. Those nodes will be hereafter called the \textbf{\boldmath$\varepsilon n$-nodes}. We shall denote by $\underline{F}^{\uparrow, ( \varepsilon n)}$ the subtree of $\underline{F}^{\uparrow}$ spanned by the $ \varepsilon n$-nodes, similarly grafted on the bi-infinite line $\R$ pointed at 0, see Figure~\ref{fig:meshing}.

 The advantage of $\underline{F}^{\uparrow, ( \varepsilon n)}$ is that it can be described by ``finitely many'' random variables (as $n \to \infty$) and that it approximates $\underline{F}^{\uparrow}$ very well in the following sense:
 
 \begin{prop}\label{prop: cv-of-epsilon-meshing-unif-in-n}
The forest $\underline{F}^{\uparrow}$ is close for the pointed Gromov-Hausdorff distance, to its $\epsilon$-meshing $\underline{F}^{\uparrow,(\epsilon n)}$, uniformly in $n$. More precisely, for every $\delta >0$ we have
\begin{equation}
\lim_{\epsilon \to 0}\limsup_{n \to \infty} \Prob{ \dgh\big( n^{-1}\cdot \underline{F}^{\uparrow}, n^{-1}\cdot  \underline{F}^{\uparrow,(\epsilon n)}\big)\geq \delta}=0.
\label{eq:cv-of-epsilon-meshing-unif-in-n}
\end{equation}
\end{prop}

We also have a convergence of $ n^{-1} \cdot \underline{F}^{\uparrow, ( \varepsilon n)}$ at fixed mesh size:

\begin{prop} For each $  \varepsilon>0$, the random forest $ n^{-1} \cdot \underline{F}^{\uparrow, ( \varepsilon n)}$ converges in distribution as $ n \to \infty$.
\label{prop:cv-fixed-mesh-size}
\end{prop}

We postpone the proofs of Propositions~\ref{prop: cv-of-epsilon-meshing-unif-in-n} and~\ref{prop:cv-fixed-mesh-size} to Sections~\ref{sec:proof-cv-of-epsilon-meshing-unif-in-n} and~\ref{sec:proof-cv-fixed-mesh-size} respectively. 

Let us explain how we can now prove Theorem~\ref{thm:main} using a double limit procedure. Let us abstract the setting for clarity. We have random variables $X_{n}$ (here $ n^{-1} \cdot  \underline{F}^{\uparrow}$) with values in a Polish space $(E,d)$ (here the space of boundedly compact pointed length spaces). For each $  \varepsilon>0$, we possess an $  \varepsilon$-approximation $ X_{n}^{{( \varepsilon)}}$ (here $n^{-1} \cdot F^{\uparrow, ( \varepsilon n)}$) which is controlled uniformly (Proposition \ref{prop: cv-of-epsilon-meshing-unif-in-n}) in the sense that 
\begin{displaymath}
  \lim_{ \varepsilon \to 0} \sup_{n \geq 1} \Delta \big( \mathcal{L}( X_{n}), \mathcal{L}(X_{n}^{{( \varepsilon)}})\big) =0,
\end{displaymath}
where $ \mathcal{L}(X)$ is the law of $X$, thus a point in the Polish space $ \mathcal{M}_{1}(E)$ of probability measures on $E$, endowed with the L\'evy-Prokhorov metric  $\Delta$. Furthermore, each $ \varepsilon$-approximation  converges in law $X_{n}^{(  \varepsilon)} \to \mathcal{X}^{(\varepsilon)}$ as $ n \to \infty$ (Proposition~\ref{prop:cv-fixed-mesh-size}), or in terms of laws using the L\'evy-Prokhorov metric:
\begin{displaymath}
  \Delta \big( \mathcal{L}(X_{n}^{(  \varepsilon)}), \mathcal{L}( \mathcal{X}^{( \varepsilon)})\big) \xrightarrow[n\to\infty]{}0.
\end{displaymath}
Coupling the previous two displays, we deduce that $(\mathcal{L}( \mathcal{X}^{( \varepsilon)}) : \varepsilon >0)$ is a Cauchy family in $(\mathcal{M}_{1}(E), \Delta)$ and so converges towards $\mathcal{L}(\mathcal{X})$ for some random variable $\mathcal{X}$. It  is then an easy matter to see that $ \Delta( \mathcal{L}(X_{n}), \mathcal{L}(\mathcal{X})) \to 0$ or equivalently that $X_{n} \to \mathcal{X}$ in distribution as $n \to \infty$. This proves the desired convergence in law.
In our case, in order to identify $ \mathcal{X}$ with the Brownian forest $ \mathcal{F}_{\sigma}$, it suffices to notice the law of the limit $ \mathcal{X}$ only depends on $ \sigma>0$, and that this is a very well-known result (see Theorem \ref{thm:aldous}) in the case when the brick distribution is associated with Bienaymé--Galton--Watson trees with finite variance! (This bypasses the technical reconstruction of the Brownian forest from the Feller flow).

\subsection{Proof of Proposition~\ref{prop: cv-of-epsilon-meshing-unif-in-n}}
\label{sec:proof-cv-of-epsilon-meshing-unif-in-n}

By definition of the local Gromov--Hausdorff topology, it is sufficient to prove the result once restricted to a ball centered at the root of arbitrary radius $A >0$ in $\underline{F}^{\uparrow}$ and $\underline{F}^{\uparrow, (\varepsilon n)}$ and we take $A=1$ to fix ideas. We first start by showing that all nodes of those trees lie in some large box (for their embedding in $ \mathbb{R}^2$) :
\begin{lem} \label{lem:box} For any $ \delta >0$, there exists $C>0$ such that for all $n$ large enough, for their natural embedding in $ \mathbb{R}^2$ all nodes of $  \mathrm{Ball}_n (\underline{F}^{\uparrow})$  lie in the ``box'' $$ \mathrm{Box}_{C,n} :=   \{(i+ \frac{1}{2},j) :  -C n \leq i \leq Cn,  0 \leq j \leq n\}$$ with probability at least  $1 -\delta$. 
\end{lem}
\begin{proof} The nodes accessible in the ball of radius $n$ centered at the root $(0,0)$ of  $\underline{F}^{\uparrow}$ belong to the nodes of $T_{i}^{\uparrow}$ at height less than $n$ for some index $i$ satisfying $-n  \leq i \leq n$. The absolute value of the abscissae of those nodes are bounded from above by the quantity
  \begin{displaymath}
    \max_{ 0 \leq k \leq n} M_{k}([-n, n[) + 1
  \end{displaymath}
which by virtue of Proposition \ref{coro:cv-of-tuple-of-pop-slices} converges in law once renormalized by $n^{-1}$. This proves the result. \end{proof}

Fix $ \delta > \varepsilon$ and let us call a node $(i+ \frac{1}{2},j)$ a \boldmath$(\delta, \varepsilon,n)$-\textbf{bad} node, if the tree descended from \unboldmath$(i+\frac{1}{2},j)$  in $F^{\uparrow}$ has height at least $\delta n$ and contains no $ \varepsilon n$-node. Clearly, if there is no $(\delta, \varepsilon,n)$-bad node within distance $n$ from the root in $F^{\uparrow}$, then every node of $  \mathrm{Ball}_n (\underline{F}^{\uparrow})$ is within distance at most $\delta n$ inside $ \underline{F}^{\uparrow}$ from an $ \varepsilon n$-node, so that the Hausdorff distance between 
$  \mathrm{Ball}_n (\underline{F}^{\uparrow})$ and $  \mathrm{Ball}_n (\underline{F}^{\uparrow, ( \varepsilon n)})$ is at most $2 \delta n$. Given the previous lemma, the proof of Proposition \ref{prop: cv-of-epsilon-meshing-unif-in-n} reduces to showing that:

\begin{lem} Fix $ \delta > \varepsilon >0$ and $C>0$. Then we have 
\begin{equation}
 \inf_{n \geq 1} \mathbb{P}\Big( \mbox{there are no $(\delta, \varepsilon,n)$-{bad} node in $ \mathrm{Box}_{C,n}$}\Big) \xrightarrow[ \varepsilon \to 0]{} 1.
\label{eq:nobad}
\end{equation}
\label{lem:nobad}\end{lem}
\begin{proof}
To prove this, notice first that  for a node to be $(\delta, \varepsilon,n)$-bad, its descendance needs first to survive for at least $ \varepsilon n$ unit of times, and then be of size less than $ \varepsilon n$ every $ \varepsilon n$ unit of times (for otherwise it would contain a $  \varepsilon n$-node) for an additional time at least $ (\delta - 2\varepsilon) n$. Combining \eqref{eq:survival-one-tree} and Corollary \ref{lem:bound proba long thin pop}, the probability of such an event is bounded above by 
\begin{equation}
 \mathrm{Cst} \cdot \frac{1}{ \varepsilon n} \cdot  \mathrm{e}^{-\delta / \varepsilon}.
\label{eq:proba-one-bad-node}
\end{equation}
Moreover, similarly to~\Cref{lem:box}, we can suppose that all points in $ \mathrm{Box}_{C,n}$ have a descendance up to level $n$ that stays within $ \mathrm{Box}_{C',n}$ for some $C'> C$, with large probability. Then, conditionally on this high probability event, if there is a $(\delta, \varepsilon,n)$-{bad} node in $ \mathrm{Box}_{C,n}$, then there must be an $(\delta - \varepsilon, \varepsilon,n)$-{bad} node among the set$$ \{ ( i+ \frac{1}{2},  j \varepsilon n) : -C' n \leq i \leq C'n , 0 \leq j \leq \varepsilon^{-1}\}.$$
 By invariance by translation (Proposition \ref{prop:invariance}), and using the bound~\eqref{eq:proba-one-bad-node}, the complement of the probability in~\eqref{eq:nobad} is therefore bounded by
$$ 2C'n \frac{1}{ \varepsilon}\cdot  \mathrm{Cst} \cdot \frac{1}{ \varepsilon n} \cdot  \mathrm{e}^{-\delta / \varepsilon} \leq \mathrm{Cst'} \frac{ \mathrm{e}^{-\delta / \varepsilon}}{ \varepsilon^{2}}.$$
For fixed $ \delta>0$, this goes to $0$ as $ \varepsilon \to 0$ (uniformly in $n$) as desired.
\end{proof}

\subsection{Proof of Proposition~\ref{prop:cv-fixed-mesh-size}}
\label{sec:proof-cv-fixed-mesh-size}

We will now prove the convergence of $ n^{-1} \cdot \underline{F}^{\uparrow, ( \varepsilon n)}$ when $\epsilon$ is fixed and $n\to \infty$. To do so, we would like to argue that, as $n\to \infty$, the forest $\underline{F}^{\uparrow, ( \varepsilon n)}$ constrained to $\mathrm{Box}_{C,n}$ is determined by a number $\BigO{\epsilon^{-2}}$ of population slices $M_k([a,b[)$. Then, Proposition~\ref{coro:cv-of-tuple-of-pop-slices} would suffice to obtain the desired convergence. The difficulty is that such a tuple of population slices does not fully determine $\underline{F}^{\uparrow, ( \varepsilon n)}$ in the box, but rather the meshing of the dual forest, $\underline{F}^{\downarrow, ( \epsilon n)}$. We then use the dual version of the $(\delta,  \varepsilon, n)$-bad nodes  and  Proposition~\ref{prop: absolute-continuity-dual} to ensure that, with high probability, the information of $\underline{F}^{\downarrow, ( \epsilon n)}$ in the box determines $\underline{F}^{\uparrow, ( \varepsilon' n)}$ for $\epsilon'>\epsilon$, up to a small error in GH distance.

Let us go into more details. Fix $  \varepsilon>0$. Using again Lemma \ref{lem:box} we can restrict ourselves to $ \mathrm{Box}_{C,n}$ for some large $C$. We will use the scaling limits of the subpopulations to describe the tree $\underline{F}^{\uparrow, ( \varepsilon n)}$ inside the box $ \mathrm{Box}_{C,n}$. Specifically, for $0<\zeta < \varepsilon$, we shall consider the joint evolution of all subpopulations between two $\zeta n$-vertices, that are  the processes 
$$ \big(M_{k}( [ (i \zeta n, j \zeta n), ((i+1) \zeta n, j \zeta n)[) : k \geq 0\big) \quad \mbox{ for } - \frac{C}{\zeta} \leq i \leq \frac{C}{\zeta}, \quad 0 \leq j \leq  \frac{1}{\zeta}.$$
By Proposition \ref{coro:cv-of-tuple-of-pop-slices}, the evolution of these processes, as well as their extinction times, admit a scaling limit which can be described using independent Feller diffusions, whose details are not crucial in what follows. 

Let us now consider the subforest $ F^{\downarrow, ( \zeta n)}$ spanned by the $\zeta n$-vertices of $  \mathrm{Box}_{C,n}$ in the descending forest $ F^{\downarrow}$ (in orange in Figure~\ref{fig:convergence} below). It is clear that this subforest is determined by the discrete subpopulations defined above and their extinction times. We claim that it is actually a continuous measurable function of these processes. To justify this claim, let us define 
\begin{align}
\hat{M}_t([(x,y),(x',y)[)&:=\frac{1}{n}M_{\lfloor nt \rfloor}([(\lfloor xn \rfloor, \lfloor yn \rfloor),(\lfloor x'n \rfloor, \lfloor yn \rfloor)[), \nonumber\\
\hat{\tau}_{(x,y),(x',y)}&:=\frac{1}{n}\tau_{(\lfloor xn \rfloor, \lfloor yn \rfloor),(\lfloor x'n \rfloor, \lfloor yn \rfloor)}.
\label{eq:interpolation-rescaled-pop-for-dual-pointed-tree}
\end{align}
We start at time $k=0$ with $2\lfloor C/\zeta \rfloor+1$ ancestral lines, at $i= -\lfloor C/\zeta \rfloor\zeta, \dots, \lfloor C/\zeta \rfloor\zeta$. Then, each time an increment of time $\zeta$ has passed:
\begin{itemize}
\item for $i=0$, the ancestral line from $(0,k\zeta)$ continues straight up to $(0,(k+1)\zeta)$;
\item for $i\geq0$, if the ancestral line from $(i,k\zeta)$ continues up to $(i',(k+1)\zeta)$, then the ancestral line from the next vertex $(j,k\zeta)$ -- where $j$ is the smallest abscissa $\ell>i$ such that there was an ancestral line at $(\ell,k\zeta)$ -- continues to the vertex $(i'\zeta+\hat{M}_{k\zeta}([(i,k\zeta),(j,k\zeta)[),(k+1)\zeta)$  -- in particular, the lines from $(i,k\zeta)$ and $(j,k\zeta)$ merge if the population between them has died out by time $(k+1)\zeta$, i.e. if $\hat{\tau}_{(i,k\zeta),(j,k\zeta)}\leq \zeta$;
\item likewise for negative $i$;
\item for any integer $-C/\zeta\leq i \leq C/\zeta$ such that the previous step does not produce a ancestral line continuing at $(i\zeta,(k+1)\zeta)$ (which is possible since the values of $\hat{M}_{(k+1)\zeta}[((i,k\zeta),(j,k\zeta))[$ are not always multiples of $\zeta$), we add a new ancestral line starting at $(i\zeta,(k+1)\zeta)$;
\end{itemize}
We proceed until time $1$ (more precisely, the last update of the ancestral lines is at time $\lfloor 1/\zeta \rfloor \zeta$, and then the lines continue upwards to $1$ without any change). This way, for $n$ large enough (so as not to have any troubles with the discretization given in~\eqref{eq:interpolation-rescaled-pop-for-dual-pointed-tree}), we have obtained the rescaled subforest $n^{-1}{F}^{\downarrow,(\zeta n)}$ from a number $(2C+1)\zeta^{-2}$ of rescaled population slices $\hat{M}_t([v,v'[)$ and their respective extinction times $(\hat{\tau}_{v,v'})$, as their image by the measurable and continuous function $G^{(\zeta)}$, which is defined as the construction above. Using~\Cref{coro:cv-of-tuple-of-pop-slices}, we deduce that it converges in law to the same function, applied to the limit tuple of $(2C+1)\zeta^{-2}$ independent Feller diffusions and their respective extinction times. 

\begin{figure}[!h]
 \begin{center}
 \includegraphics[width=12cm]{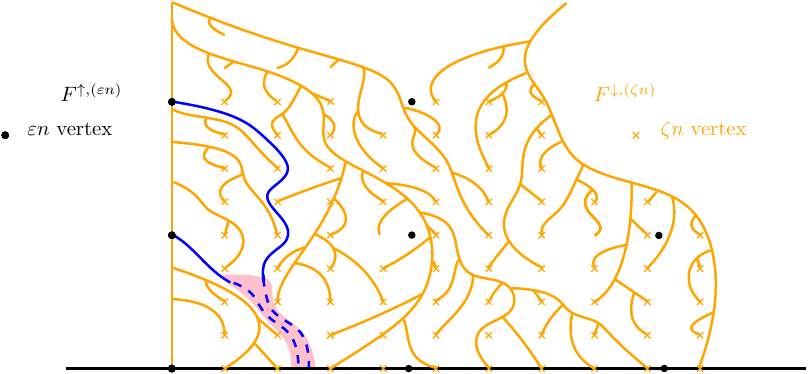}
 \caption{Illustration of the control of $ F^{\uparrow, (\varepsilon n)}$ by $ F^{\downarrow, (\zeta n )}$. As soon as  two branches of $ F^{\uparrow, (\varepsilon n)}$ are in the same interval (in red) of $ F^{\downarrow, (\zeta n )}$ they must meet within $\delta n$ if there is no $( \delta, \zeta, n)$ bad node.} \label{fig:convergence}
 \end{center}
 \end{figure}
 To simplify the reasoning, let us suppose that $  \varepsilon/\zeta \notin \mathbb{Q}$. Because of this, and since the Feller diffusions have densities at all times, none of the points in $ F^{\downarrow, ( \zeta n)}$ can be identified with an $ \varepsilon n$-node in the scaling limit. Clearly, the knowledge of  $ F^{\downarrow, ( \zeta n)}$ gives some constraint on $F^{\uparrow, ( \varepsilon n)}$ inside $ \mathrm{Box}_{C,n}$, but not enough to control its metric structure. The problem comes from the fact, although we know when two ancestral linages of some $ \varepsilon n$-node are in the same ``interval'' (in pink in Figure \ref{fig:convergence}), we do not know whether those two lineages meet or not. However, it is easy to see that two such lineages must meet within distance $\delta n$ as soon as there is no \emph{dual} $(\delta, \zeta, n)$-bad node in the box, i.e. a bad mode of the descending forest. The fact that there is no such bad point in the box with high probability is  implied by the contiguity of the descending forest with respect of the ascending ones for the reversed block model (Proposition~\ref{prop: absolute-continuity-dual})  and the previous Lemma~\ref{lem:nobad}.

 We can then conclude that with high probability, the subforest $ F^{\uparrow, ( \varepsilon n)}$ (restricted to the box) is known up to $ \delta n$ in Gromov--Hausdorff distance using $ F^{\downarrow, ( \zeta n)}$. Since the latter converges in distribution for any fixed $ \zeta>0$, we get the result.

\section{Extensions}
\label{sec:extensions}

In this section, we discuss extensions of the model that we have studied in this paper, in two different directions. The first one goes beyond the Brownian regime, in the case where the brick distribution $\rho$ is heavy-tailed. The second one is a continuous-time model, where the bricks now have random, continuum heights.

\subsection{Stable brick distributions}

If we consider brick distributions with heavy tails, we might expect two different types of asymptotic behavior. Indeed, suppose that $\beta$ and $\theta$ are respectively in the domain of attraction of an $\alpha$- (resp.\ $\alpha'$-) stable law, with $1<\alpha,\alpha'\leq2$. Then, if, say $\alpha'<\alpha$, we expect that the fluctuations of $\theta$ prevail, so that, asympotically, the associated trees should essentially behave like independent $\alpha'$-stable trees. Similarly, if $\alpha'>\alpha$, then the dual trees should behave asymptotically like independent $\alpha$-stable trees. Now, if both $\beta$ and $\theta$ have $\alpha$-stable tails, then neither prevails on the other, and we expect to see a new type of limiting object, related to Lévy processes with both positive and negative jumps.

The study of either of these regimes requires more technical arguments than the Brownian regime. Indeed, in those cases, the lack of second moments for $\rho$ entails that the size-biased bottom of the root brick has infinite expectation, and it also prevents us from using the arguments of the present paper to prove that subpopulations are martingales. In a work in progress, we are tackling those hurdles to prove the expected limit behaviors.

\subsection{Continuous-time model}

Let us consider a variant of the case where $\rho=\beta\otimes\eta$, that we describe informally, similarly to~\Cref{subsec:local-cata-bgw}. Let $\mu$ and $\nu$ be probability distributions on $\Z_{\geq2}, \Z_{\geq1}$ respectively, with finite third moments, and with respective expectations $Z_{\mu},Z_{\nu}-1$. We start at time $t=0$ with countably many individuals, represented by the nodes on $\Z$, and we equip them with independent pairs of independent exponential clocks $(X,X')$ of respective parameters 1 and $\lambda=Z_{\mu}/Z_{\nu}$. Then, when the first clock rings, if it is one of the first type, the corresponding individual gives birth to children according to $\mu$; if it is one the second type, a number of individuals on its right, starting with itself, die according to $\nu$. We continue this recursively on the descendants of the original individuals. (Note that a.s. no two clocks ring simultaneously.) This yields a bi-infinite forest $\mathscr{F}$, that we transform into a pointed tree $\underline{\mathscr{F}}$ as before, by grafting it on the real line and pointing it at 0.

Let us now sketch the arguments to transfer the results that we have proven on discrete brick walls to this new variant. We can construct a coupling between $\underline{\mathscr{F}}$ and a sequence $\big(\underline{F}^{\uparrow (N)} : N \geq 1\big)$ of forests with discrete edge lengths, induced by brick walls, such that $\underline{\mathscr{F}}$ is well approximated by $\underline{F}^{\uparrow (N)}$, as $N \to \infty$. Let $N$ be a positive integer such that $1+\lambda<N$, and let $\ell_N:=\ln{\frac{N}{N-1}}$. For each individual $i$ in $\mathscr{F}$, we define the following discrete approximations of the clocks $(X^{(i)},X'^{(i)})$ associated to~$i$:
\[
Y^{(i)}:=\ell_N\Big\lfloor \frac{X^{(i)}}{\ell_N}\Big\rfloor, \ \ Y'^{(i)}:=\ell_N\Big\lfloor \frac{X'^{(i)}}{\ell_N}\Big\rfloor.
\]
We then pick the exact same birth/death events after each discrete waiting time, as for their continuum counterparts: this yields a forest ${F}^{\uparrow (N)}$ where edge lengths are multiples of $\ell_N$.

By construction, $Y^{(i)}/\ell_N$ and $Y'^{(i)}/\ell_N$ are geometric variables of parameters $1-1/N$ and $1-\lambda/N$ respectively. Thus, the forest ${F}^{\uparrow (N)}$ has the same law (up to rescaling time by $\ell_N$) as the primal forest of a brick wall model whose brick distribution $\rho_N$ is as follows:
\begin{itemize}
\item $\rho_N(1,k) =(1/N)\mu(k), \, \forall\, k\geq2$
\item $\rho_N(k,1) =(\lambda/N)\nu(k-1), \, \forall\, k\geq2$
\item $\rho_N(1,1)= 1-1/N-\lambda/N$.
\end{itemize}
This model is critical, with both marginals $\eta_N$ and $\beta_N$ having finite third moments. We compute the associated variance $\sigma_N$ given by~\eqref{eq:sigma-definition}:
\[
\sigma_N^2=\frac{\Expect{\mu^2}-2\Expect{\mu}+\lambda(\Expect{(\nu+1)^2}-2\Expect{\nu+1})}{N+\Expect{\mu}-1-\lambda}.
\]
By the results of the previous sections, if we take bricks with height 1 as before, for fixed $N$ the discrete primal forest, rescaled by $1/n$, converges as $n \to \infty$ to the Brownian forest of parameter $\sigma_N$. As the actual the discrete-time approximation $\underline{F}^{\uparrow (N)}$ of $\underline{\mathscr{F}}$ is rescaled in time by $\ell_N$, for fixed $N$, $\frac{1}{n}\underline{F}^{\uparrow (N)}$ thus converges as $n \to \infty$ to a Brownian forest of parameter $\sigma_N/\sqrt{\ell_N}$. We see that this scaling parameter behaves like $\sigma=\left[\Expect{\mu^2}-2\Expect{\mu}+\lambda(\Expect{(\nu+1)^2}-2\Expect{(\nu+1)})\right]^{1/2}$, as $N\to\infty$.

Therefore, we expect that $\frac{1}{n}\underline{\mathscr{F}}$ itself converges to a Brownian forest of parameter $\sigma$. To prove this, we can establish a bound on the Gromov--Hausdorff distance between $\frac{1}{n}\mathrm{Ball}_n(\underline{\mathscr{F}})$ and $\frac{1}{n}\mathrm{Ball}_n(\underline{F}^{\uparrow (N)})$, going to zero as $N \to \infty$, uniformly in $n$. We have a natural correspondence between $\underline{\mathscr{F}}$ and $\underline{F}^{\uparrow (N)}$: from the definition of the coupling, they have the same exact genealogical structure, and differences in edge lengths of at most $\ell_N$ (with full probability), which induces a correspondence $\mathcal{R}$ where, for $x,y \in \underline{F}^{\uparrow (N)}, x',y' \in \underline{\mathscr{F}}$, and $(x,x'), (y,y') \in \mathcal{R}$: 
\[d(x',y') \cdot (1-\ell_N) \leq d(x,y) \leq d(x',y').\]
Thus, for any $n$, the distortion of this correspondence, restricted to balls of radius $n$, rescaled by $1/n$, is smaller than $\ell_N$. This suffices to conclude.

\printbibliography

\end{document}